\documentclass[11pt]{amsart}
\usepackage[margin=1in]{geometry}
\usepackage{amsmath,amssymb,amsthm}
\usepackage{tikz}
\usepackage{booktabs}
\usepackage{enumitem}
\usepackage{xcolor}
\usepackage[colorlinks=true,linkcolor=blue!50!black,citecolor=green!40!black,urlcolor=blue,bookmarksnumbered]{hyperref}

\theoremstyle{plain}
\newtheorem{theorem}{Theorem}[section]
\newtheorem{lemma}[theorem]{Lemma}
\newtheorem{proposition}[theorem]{Proposition}
\newtheorem{corollary}[theorem]{Corollary}
\theoremstyle{definition}
\newtheorem{definition}[theorem]{Definition}
\newtheorem{example}[theorem]{Example}
\theoremstyle{remark}
\newtheorem{remark}[theorem]{Remark}

\begin{document}

\title[Weakly contractible non-contractible spaces of ten points]{On weakly contractible non-contractible finite topological spaces of ten points}

\author{Ponaki Das}
\address{Department of Mathematics, North-Eastern Hill University, Shillong 793022, India}
\email{ponaki.das20@gmail.com}

\author{Sainkupar Marwein Mawiong}
\address{Department of Basic Sciences and Social Sciences, North-Eastern Hill University, Shillong 793022, India}
\email{skupar@gmail.com}



\begin{abstract}
Cianci and Ottina proved that a homotopically trivial non-contractible finite $T_0$-space has at least nine points, and classified such spaces with exactly nine points. We complete the classification for ten points. The main tool is the notion of a naked pair: a relation $c<a$ between a minimal and a maximal point that does not factor through a middle point. We show that minimal homotopically trivial spaces with at most ten points have no naked pairs. This turns minimality into simple conditions on the incidences between the three levels of the space and, together with two lemmas of Cianci and Ottina and an Euler characteristic count, reduces the classification to short case analyses. No such space exists with one or two middle points; there are exactly six with three middle points, forming three types and their order-duals, and exactly four with four middle points. All ten spaces have height two, an antichain of middle points and collapsible order complexes, and the count is confirmed by an exhaustive computer enumeration. We also show that the absence of naked pairs persists up to seventeen points but not beyond, and that height three already occurs at eleven points.
\end{abstract}

\maketitle

\section{Introduction}\label{sec:intro}

Alexandroff~\cite{Alexandroff1937} established a bijective correspondence between finite $T_0$-spaces and finite partially ordered sets. Stong~\cite{Stong1966} introduced beat points and showed that every finite $T_0$-space deformation retracts to a unique minimal core, and McCord~\cite{McCord1966} constructed a weak homotopy equivalence from the geometric realisation $|\mathcal{K}(X)|$ of the order complex of $X$ to $X$, linking the combinatorial structure of $X$ to its classical homotopy theory. For foundational results on minimal finite models and on strong homotopy types of finite spaces, we refer to Barmak and Minian~\cite{BarmakMinian2007,BarmakMinian2012}.

A finite $T_0$-space $X$ has two notions of triviality: contractibility, which by Stong's theorem is detected by successively removing beat points, and homotopical triviality (equivalently, weak contractibility), which by McCord's and Whitehead's theorems holds exactly when $|\mathcal{K}(X)|$ is contractible. A space with more than one point and no beat points is never contractible, yet it may be homotopically trivial. Cianci and Ottina showed that such a space has at least nine points~\cite[Theorem~3.6]{CianciOttina2020}, and that there are exactly two such spaces with nine points~\cite[Theorem~3.7]{CianciOttina2020}. This paper determines all such spaces with ten points.

The points of a connected finite $T_0$-space with more than one point split into its minimal points, its maximal points, and the remaining \emph{middle points} $B_X$. The ten-point examples fall into two families, according to whether $|B_X|$ is $3$ or $4$, and the classification is organised by $|B_X|$. Its main new ingredient is Theorem~\ref{thm:nonaked}: in a minimal homotopically trivial space with at most ten points, every relation between a minimal and a maximal point factors through a middle point. This turns minimality into simple conditions on the incidences between the three levels, and together with two lemmas of Cianci and Ottina and an Euler characteristic count it reduces each case to a short argument.

\begin{theorem}[Main Theorem]\label{thm:main}
	Let $X$ be a homotopically trivial non-contractible minimal finite
	$T_0$-space with $|X|=10$, and let
	$B_X=X\setminus\bigl(\operatorname{mnl}(X)\cup\operatorname{mxl}(X)\bigr)$.
	Then $1\leq|B_X|\leq 4$, and the following hold.
	\begin{enumerate}[label=\textup{(\arabic*)}]
		\item No such space exists with $|B_X|=1$.
		\item No such space exists with $|B_X|=2$.
		\item There are exactly six such spaces with $|B_X|=3$, up to
		homeomorphism: the spaces of Types~\textup{I}, \textup{II} and
		\textup{III}, which satisfy $|\operatorname{mnl}(X)|=3$ and
		$|\operatorname{mxl}(X)|=4$, and their order-duals.
		\item There are exactly four such spaces with $|B_X|=4$, up to
		homeomorphism: the spaces of Types~\textup{IV}, \textup{V},
		\textup{VI} and \textup{VII}, all with
		$|\operatorname{mnl}(X)|=|\operatorname{mxl}(X)|=3$. Types~\textup{IV}
		and~\textup{VII} are self-dual, and Types~\textup{V} and~\textup{VI}
		are order-duals of each other.
	\end{enumerate}
	All ten spaces have height two, $B_X$ is an antichain in each of them,
	and their order complexes are collapsible. Moreover, in each of them
	every relation $c<a$ with $c$ minimal and $a$ maximal factors through
	$B_X$.
\end{theorem}

For a minimal space with more than one point, non-contractibility is automatic (by the Classification Theorem~\cite[Corollary~1.3.7]{Barmak2011}; see Section~\ref{subsec:finite}); it is kept in the statement to emphasise the gap studied here. The last assertion of the theorem holds more generally for every minimal homotopically trivial space with at most ten points (Theorem~\ref{thm:nonaked}).

Section~\ref{sec:prelim} fixes notation, collects the results of Barmak--Minian and Cianci--Ottina that we use, and proves the results on naked pairs (Definition~\ref{def:naked}). Section~\ref{sec:bx12} treats the cases $|B_X|\leq2$, and Sections~\ref{sec:bx3} and~\ref{sec:bx4} the cases $|B_X|=3$ and $|B_X|=4$; the proof of Theorem~\ref{thm:main} is assembled at the end of Section~\ref{sec:bx4}. Section~\ref{sec:computation} describes an independent computer enumeration, and Section~\ref{sec:conclusion} concludes; there we also show that Theorem~\ref{thm:nonaked} holds for all minimal homotopically trivial spaces with at most seventeen points, and that this bound is sharp.

\section{Preliminaries}\label{sec:prelim}

\subsection{Finite spaces}\label{subsec:finite}

We follow the notation of Barmak~\cite{Barmak2011}; see also~\cite{May2003}. A topological space is $T_0$ if for any two distinct points some open set contains exactly one of them. By Alexandroff's theorem~\cite{Alexandroff1937}, finite $T_0$-spaces correspond bijectively to finite posets via the \emph{specialisation order}, in which $x\leq y$ if every open set containing $y$ also contains $x$. The open sets are then exactly the down-sets, continuous maps correspond to order-preserving maps~\cite[Prop.~1.2.1]{Barmak2011}, and homeomorphisms correspond to order isomorphisms. We identify finite $T_0$-spaces with their posets. For $x\in X$ we write
\begin{gather*}
U_x=\{y\in X:y\leq x\},\qquad F_x=\{y\in X:y\geq x\},\\
\hat U_x=U_x\setminus\{x\},\qquad \hat F_x=F_x\setminus\{x\};
\end{gather*}
thus $U_x$ is the smallest open set containing $x$, and $F_x$ is the closure of $\{x\}$. We write $x<y$ if $x\leq y$ and $x\neq y$; then $y$ \emph{covers} $x$ if $x<y$ and no point $z$ satisfies $x<z<y$. We write $\operatorname{mnl}(X)$ and $\operatorname{mxl}(X)$ for the sets of minimal and maximal points of $X$. The \emph{opposite} (or \emph{order-dual}) space $X^{\mathrm{op}}$ has the same points and the reversed order. The \emph{non-Hausdorff join} $X\circledast Y$ of two finite $T_0$-spaces is their disjoint union, with the given orders on $X$ and on $Y$ and with $x<y$ for all $x\in X$ and $y\in Y$. The \emph{non-Hausdorff suspension} of $X$ is $\mathbb{S}(X)=X\circledast S^0$, where $S^0$ is the discrete two-point space, and $\mathbb{S}^hS^0$ denotes its $h$-fold iterate~\cite[Section~2.7]{Barmak2011}.

The \emph{order complex} $\mathcal{K}(X)$ is the simplicial complex whose simplices are the non-empty chains of $X$. A chain $x_0<x_1<\dots<x_k$ has \emph{length} $k$, and the \emph{height} $\mathrm{h}(X)$ is the maximum length of a chain in $X$, so that $\mathrm{h}(X)=\dim\mathcal{K}(X)$. Since $X$ and $X^{\mathrm{op}}$ have the same chains, $\mathcal{K}(X^{\mathrm{op}})=\mathcal{K}(X)$. McCord's theorem~\cite[Theorem~2]{McCord1966} gives a weak homotopy equivalence $|\mathcal{K}(X)|\to X$; in particular, the singular homology of $X$ is the simplicial homology of $\mathcal{K}(X)$. Homology is taken with integer coefficients.

A point $x$ is a \emph{down-beat point} if $\hat U_x$ has a maximum, and an \emph{up-beat point} if $\hat F_x$ has a minimum~\cite[Def.~1.3.3]{Barmak2011}; equivalently, $x$ covers exactly one point, respectively is covered by exactly one point. If $x$ is a beat point, then $X\setminus\{x\}$ is a strong deformation retract of $X$~\cite[Prop.~1.3.4]{Barmak2011}. A space without beat points is called \emph{minimal} (a \emph{minimal finite space}, or a \emph{core}). By Stong's theorem~\cite{Stong1966}, removing beat points one at a time reduces $X$ to a minimal space, its \emph{core}, which is unique up to homeomorphism; moreover, two finite $T_0$-spaces are homotopy equivalent if and only if their cores are homeomorphic~\cite[Theorem~4]{Stong1966}. In particular, by the Classification Theorem~\cite[Corollary~1.3.7]{Barmak2011} (see also~\cite[Corollary~4]{Stong1966}), a finite $T_0$-space is \emph{contractible}, that is, homotopy equivalent to a point, if and only if its core is a point. This is the notion of contractibility used by Cianci and Ottina~\cite{CianciOttina2020}; in particular, a minimal space with more than one point is not contractible.

A space $X$ is \emph{homotopically trivial}, or \emph{weakly contractible}, if it is non-empty and path-connected and all its homotopy groups vanish. By McCord's theorem and Whitehead's theorem~\cite[Theorem~1]{Whitehead1949}, a finite $T_0$-space $X$ is homotopically trivial exactly when $|\mathcal{K}(X)|$ is contractible. This paper classifies the minimal homotopically trivial spaces with ten points, which are therefore not contractible.

Finally, an \emph{elementary collapse} of a simplicial complex removes a simplex $\sigma$ together with a proper face $\tau$ of $\sigma$ that lies in no other simplex (a \emph{free face} of $\sigma$). A complex is \emph{collapsible} if a sequence of elementary collapses reduces it to a point~\cite[p.~52]{Barmak2011}, and a collapsible complex is contractible.

\subsection{Layers and degrees}\label{subsec:notation}

\begin{definition}\label{def:layers}
	Let $X$ be a finite $T_0$-space. Write $C=\operatorname{mnl}(X)$,
	$A=\operatorname{mxl}(X)$ and
	\[
	B=B_X=X\setminus(C\cup A),
	\]
	whose elements are the \emph{middle points} of $X$, and put $m=|C|$, $r=|B|$ and
	$n=|A|$. For $x\in X$ and $S\subseteq X$ write
	\[
	S_{<x}=\{y\in S:y<x\},\qquad S_{>x}=\{y\in S:y>x\}.
	\]
	For $b\in B$, $c\in C$ and $a\in A$ set
	\[
	\beta(b)=|C_{<b}|,\qquad \alpha(b)=|A_{>b}|,\qquad t(c)=|B_{>c}|,\qquad s(a)=|B_{<a}|,
	\]
	and let $N_{CA}=|\{(c,a)\in C\times A:c<a\}|$. When the middle points
	are indexed as $B=\{b_1,\dots,b_r\}$ we write $\beta_j=\beta(b_j)$ and
	$\alpha_j=\alpha(b_j)$.
\end{definition}

Suppose that $X$ is connected and $|X|>1$. A point that is both minimal and maximal is comparable with no other point, and is therefore an isolated point; so $X=C\sqcup B\sqcup A$, where $C$ and $A$ are antichains. Note that $\beta(b)$ counts all points of $C$ below $b$, not only those covered by $b$. For $a\in A$ the set $B_{<a}$ is a down-set of the subposet $B$, and for $c\in C$ the set $B_{>c}$ is an up-set of $B$. Moreover $b\in B_{<a}$ if and only if $a\in A_{>b}$, and $b\in B_{>c}$ if and only if $c\in C_{<b}$.

Passing to the opposite space $X^{\mathrm{op}}$ exchanges $C$ and $A$, $\beta$ and $\alpha$, and $s$ and $t$, reverses the order of $B$, and leaves $N_{CA}$ unchanged. Since $\mathcal{K}(X^{\mathrm{op}})=\mathcal{K}(X)$ and the up-beat points of $X^{\mathrm{op}}$ are the down-beat points of $X$ and vice versa, $X^{\mathrm{op}}$ is minimal, respectively homotopically trivial, if and only if $X$ is. Each statement below therefore has an order-dual, which we use without further comment. Finally, a homeomorphism of finite $T_0$-spaces is the same as an order isomorphism, so it maps $C$, $B$ and $A$ onto the corresponding sets.

\begin{lemma}[{\cite[Proposition~2.3(b)]{CianciOttina2020}}]\label{lem:two-extremal}
	Let $X$ be a minimal finite $T_0$-space. If $x\notin\operatorname{mxl}(X)$
	then $|\hat F_x\cap\operatorname{mxl}(X)|\geq 2$, and if
	$x\notin\operatorname{mnl}(X)$ then $|\hat U_x\cap\operatorname{mnl}(X)|\geq2$.
	In particular, if $X$ is connected and $|X|>1$, then $\beta(b)\geq2$ and
	$\alpha(b)\geq2$ for every $b\in B$, $|A_{>c}|\geq 2$ for every $c\in C$,
	and $|C_{<a}|\geq2$ for every $a\in A$.
\end{lemma}

\begin{proof}
	Let $x\notin\operatorname{mxl}(X)$, so that $\hat F_x\neq\emptyset$. A maximal element of the subposet $\hat F_x$ is maximal in $X$, so $\hat F_x\cap\operatorname{mxl}(X)\neq\emptyset$. Suppose that $\hat F_x\cap\operatorname{mxl}(X)=\{a\}$. Every $y\in\hat F_x$ lies below some maximal point of $X$, which then belongs to $\hat F_x\cap\operatorname{mxl}(X)$; hence $y\leq a$. If $\hat F_x=\{a\}$, then $x$ is an up-beat point. Otherwise choose a maximal element $y$ of $\hat F_x\setminus\{a\}$. Then $\hat F_y\subseteq\hat F_x$ contains no point of $\hat F_x\setminus\{a\}$, by the choice of $y$, and contains $a$; so $\hat F_y=\{a\}$ and $y$ is an up-beat point. Both cases contradict minimality. The second assertion is the first one applied to $X^{\mathrm{op}}$, and the last sentence follows because $C\cap A=\emptyset$ when $X$ is connected and $|X|>1$.
\end{proof}

\subsection{Results of Barmak--Minian and Cianci--Ottina}\label{subsec:cited}

We record the results from~\cite{BarmakMinian2007,CianciOttina2020} that are used below, in the notation of Definition~\ref{def:layers}.

\begin{theorem}[{Barmak--Minian~\cite[Theorem~2.12]{BarmakMinian2007}; see also \cite[Theorem~3.2.1]{Barmak2011} and \cite[Theorem~2.1]{CianciOttina2020}}]\label{thm:BM}
	Let $X$ be a minimal finite $T_0$-space with $|X|>1$. Then
	$|X|\geq 2\,\mathrm{h}(X)+2$, and if equality holds then $X$ is
	homeomorphic to the $\mathrm{h}(X)$-fold non-Hausdorff suspension
	$\mathbb{S}^{\mathrm{h}(X)}S^0$ of the discrete two-point space $S^0$.
\end{theorem}

In~\cite{BarmakMinian2007} the height of a poset is the number of points of a longest chain, which is $\mathrm{h}(X)+1$; with that convention the bound reads $|X|\geq2\,\mathrm{h}(X)$.

\begin{lemma}[{\cite[Lemma~3.1]{CianciOttina2020}}]\label{lem:L31}
	Let $X$ be a minimal finite $T_0$-space and let $b<a$ in $X$. Then
	$|\hat{U}_a|\geq|\hat{U}_b|+2$ and $|\hat{F}_b|\geq|\hat{F}_a|+2$.
\end{lemma}

\begin{proof}
	Since $b<a$, we have $U_b\subseteq\hat U_a$. If equality held, $b$ would be the maximum of $\hat U_a$ and $a$ would be a down-beat point. Hence $|\hat U_a|\geq|U_b|+1=|\hat U_b|+2$. The second inequality is the first one applied to $X^{\mathrm{op}}$.
\end{proof}

\begin{lemma}[{\cite[Lemma~3.3]{CianciOttina2020}}]\label{lem:L33}
	Let $X$ be a homotopically trivial finite $T_0$-space with
	$\mathrm{h}(X)=2$, and let $b\neq b'$ be middle points of $X$. If
	$|A_{>b}\cap A_{>b'}|\geq2$, then $|C_{<b}\cap C_{<b'}|\leq1$.
\end{lemma}

In~\cite{CianciOttina2020} the lemma is stated as $|\hat F_b\cap\hat F_{b'}|\geq 2\Rightarrow|\hat U_b\cap\hat U_{b'}|\leq1$. Since $\mathrm{h}(X)=2$, two middle points are incomparable, so $\hat F_b\cap\hat F_{b'}=A_{>b}\cap A_{>b'}$ and $\hat U_b\cap\hat U_{b'}=C_{<b}\cap C_{<b'}$. Equivalently, if $|C_{<b}\cap C_{<b'}|\geq2$ then $|A_{>b}\cap A_{>b'}|\leq1$.

\begin{lemma}[{\cite[Lemma~3.4]{CianciOttina2020}}]\label{lem:L34}
	Let $Y$ be a topological space such that $\pi_1(Y,y_0)$ is not a
	non-trivial perfect group for some $y_0\in Y$. If $\mathbb{S}(Y)$ is
	homotopically trivial, then $Y$ is homotopically trivial.
\end{lemma}

\begin{proposition}[{\cite[Proposition~3.5]{CianciOttina2020}}]\label{prop:P35}
	A homotopically trivial finite $T_0$-space $X$ with $\mathrm{h}(X)\leq1$ is
	contractible.
\end{proposition}

\begin{theorem}[{\cite[Theorem~3.6]{CianciOttina2020}}]\label{thm:CO}
	Every homotopically trivial non-contractible topological space has at
	least nine points.
\end{theorem}

\subsection{The Euler characteristic}\label{subsec:euler3}

\begin{lemma}\label{lem:euler3}
	Let $X$ be a connected finite $T_0$-space with $|X|>1$ and
	$\mathrm{h}(X)\leq3$, and let $R_B=\{(b,b')\in B\times B:b<b'\}$. Then
	\begin{equation}\label{eq:2}
		\chi(\mathcal{K}(X))=m+n-N_{CA}+\sum_{b\in B}(\beta(b)-1)(\alpha(b)-1)
		-\sum_{(b,b')\in R_B}(\beta(b)-1)(\alpha(b')-1).
	\end{equation}
\end{lemma}

\begin{proof}
	A chain of $X$ contains at most one point of $C$ and at most one point of $A$. Its middle points form a chain with at most two elements, since a chain of three middle points could be extended by a point of $C$ below and a point of $A$ above to a chain of five points. Let $f_k$ be the number of $k$-simplices of $\mathcal{K}(X)$. Counting chains by type,
	\[
	f_0=m+r+n,\qquad f_1=\sum_{b}\beta(b)+|R_B|+\sum_{b}\alpha(b)+N_{CA},
	\]
	\[
	f_2=\sum_{b}\beta(b)\alpha(b)+\sum_{(b,b')\in R_B}\bigl(\beta(b)+\alpha(b')\bigr),\qquad
	f_3=\sum_{(b,b')\in R_B}\beta(b)\alpha(b'),
	\]
	and $f_k=0$ for $k\geq4$. For instance, the $2$-simplices $c<b<a$ contribute $\beta(b)\alpha(b)$, those of the form $c<b<b'$ contribute $\beta(b)$, and those of the form $b<b'<a$ contribute $\alpha(b')$. Substituting into $\chi(\mathcal{K}(X))=f_0-f_1+f_2-f_3$ and using the identities
	\[
	-\beta-\alpha+\beta\alpha=(\beta-1)(\alpha-1)-1,\qquad
	-1+\beta+\alpha'-\beta\alpha'=-(\beta-1)(\alpha'-1)
	\]
	for the terms indexed by $b\in B$ and by $(b,b')\in R_B$ respectively gives~\eqref{eq:2}.
\end{proof}

\begin{corollary}\label{cor:euler2}
	If $X$ is as in Lemma~\ref{lem:euler3}, $X$ is homotopically trivial
	and $B$ is an antichain, then
	\[
		N_{CA}=m+n-1+\sum_{b\in B}(\beta(b)-1)(\alpha(b)-1).
	\]
\end{corollary}

\begin{proof}
	Here $R_B=\emptyset$ and $\chi(\mathcal{K}(X))=1$.
\end{proof}

\subsection{First reductions}\label{subsec:reductions}

In this subsection $X$ is a minimal homotopically trivial finite $T_0$-space with $|X|=10$. Being homotopically trivial, $X$ is path-connected, so Definition~\ref{def:layers} gives a partition $X=C\sqcup B\sqcup A$; being minimal with more than one point, $X$ is not contractible.

\begin{proposition}\label{prop:height}
	$\mathrm{h}(X)\in\{2,3\}$.
\end{proposition}

\begin{proof}
	If $\mathrm{h}(X)\leq1$, then $X$ is contractible by Proposition~\ref{prop:P35}, which is not the case. By Theorem~\ref{thm:BM}, $10\geq 2\,\mathrm{h}(X)+2$, so $\mathrm{h}(X)\leq 4$; and if $\mathrm{h}(X)=4$, equality holds and $X\cong\mathbb{S}^4S^0$. By McCord's theorem on non-Hausdorff suspensions~\cite[Theorem~5]{McCord1966} (see also \cite[Proposition~3.1.4]{Barmak2011}), $\mathbb{S}^4S^0$ is weakly homotopy equivalent to the sphere $S^4$, which is not homotopically trivial. Hence $\mathrm{h}(X)\leq3$.
\end{proof}

\begin{lemma}\label{lem:size}
	$|C|\geq3$ and $|A|\geq 3$. Consequently $1\leq|B|\leq4$.
\end{lemma}

\begin{proof}
	We prove $|A|\geq3$; applying this to $X^{\mathrm{op}}$ gives $|C|\geq 3$. If $|A|=1$, then $X$ has a maximum and is contractible, which is not the case. Suppose that $|A|=2$. By \cite[Proposition~2.3(c)]{CianciOttina2020}, $X$ is homeomorphic to $\mathbb{S}(X')$, where $X'=X\setminus A$ has eight points. The space $X'$ has no beat points: for $x\in X'$ we have $\hat U^{X}_x=\hat U^{X'}_x$ and $\hat F^{X}_x=\hat F^{X'}_x\cup A$, where both points of $A$ lie above every point of $X'$, so a maximum of $\hat U^{X'}_x$ would be a maximum of $\hat U^{X}_x$, and a minimum of $\hat F^{X'}_x$ would be a minimum of $\hat F^{X}_x$. By \cite[Theorem~5.7]{CianciOttina2018}, the fundamental group of a finite $T_0$-space with at most twelve points is free, so it is not a non-trivial perfect group, and Lemma~\ref{lem:L34} shows that $X'$ is homotopically trivial. Being minimal with more than one point, $X'$ is not contractible, by the Classification Theorem~\cite[Corollary~1.3.7]{Barmak2011}; this contradicts Theorem~\ref{thm:CO}. Hence $|A|\geq3$.
	
	\sloppy Finally, $|B|=10-|C|-|A|\leq4$, and $|B|\geq1$ because $\mathrm{h}(X)\geq2$ by Proposition~\ref{prop:height}: the middle element of a chain $x<y<z$ lies in $B$.
\end{proof}

\subsection{Naked pairs}\label{subsec:naked}

In this subsection $X$ is a connected finite $T_0$-space with $|X|>1$, so that $X=C\sqcup B\sqcup A$. Its \emph{comparability graph} $\mathcal{G}(X)$ is the $1$-skeleton of $\mathcal{K}(X)$: its vertices are the points of $X$, and its edges are the pairs of comparable points. It is connected, because $X$ is: the set of vertices of a component of $\mathcal{G}(X)$ is both a down-set and an up-set of $X$, hence both open and closed.

\begin{definition}\label{def:naked}
	A pair $(c,a)\in C\times A$ with $c<a$ is \emph{transitive} if
	$B_{>c}\cap B_{<a}\neq\emptyset$, that is, if $c<b<a$ for some $b\in B$,
	and \emph{naked} otherwise. We write $E_X$ for the set of naked pairs.
	Let $\mathcal{G}_B(X)$ be the spanning subgraph of $\mathcal{G}(X)$ whose
	edges are the comparable pairs with at least one element in $B$, and let
	$\kappa(X)$ be the number of connected components of $\mathcal{G}_B(X)$. Finally, let $\mathcal{Q}(X)$ be the multigraph, with loops allowed, whose vertices are the connected components of $\mathcal{G}_B(X)$, with one edge joining the component of $c$ to the component of $a$ for each naked pair $(c,a)$.
\end{definition}

Since $C$ and $A$ are antichains, every edge of $\mathcal{G}(X)$ is either an edge of $\mathcal{G}_B(X)$, a transitive pair, or a naked pair.

\begin{lemma}\label{lem:naked}
	Let $(c,a)$ be a naked pair. Then $\{c,a\}$ is a face of no $2$-simplex of
	$\mathcal{K}(X)$. Consequently, if some closed edge-path in
	$\mathcal{G}(X)$ traverses the edge $\{c,a\}$ exactly once, then
	$H_1(\mathcal{K}(X))\neq0$.
\end{lemma}

\begin{proof}
	A $2$-simplex containing $c$ and $a$ is a chain $x<y<z$ with $\{c,a\}\subseteq\{x,y,z\}$. As $c$ is minimal and $a$ is maximal, $x=c$, $z=a$ and $y\in B$, so $(c,a)$ would be transitive. For the second statement, let $v_0,v_1,\dots,v_k=v_0$ be the closed edge-path and consider the $1$-chain $z=\sum_{i=0}^{k-1}[v_i,v_{i+1}]$, with the convention $[y,x]=-[x,y]$. Then $\partial z=\sum_{i}(v_{i+1}-v_i)=0$, and the coefficient of $[c,a]$ in $z$ is $\pm1$. On the other hand, the coefficient of $[c,a]$ in the boundary of a $2$-chain is a signed sum of the coefficients of $2$-simplices having $\{c,a\}$ as a face, hence is $0$. So $z$ is a cycle that is not a boundary.
\end{proof}

\begin{theorem}\label{thm:tree}
	Let $X$ be a connected finite $T_0$-space with $|X|>1$ and
	$H_1(\mathcal{K}(X))=0$. Then $\mathcal{Q}(X)$ is a tree. In particular
	$|E_X|=\kappa(X)-1$.
\end{theorem}

\begin{proof}
	If $(c,a)$ is transitive, say $c<b<a$ with $b\in B$, then $c,b,a$ is a path in $\mathcal{G}_B(X)$, so $c$ and $a$ lie in the same component. Hence every edge of $\mathcal{G}(X)$ that is not naked joins two vertices of the same component of $\mathcal{G}_B(X)$, and since $\mathcal{G}(X)$ is connected, so is $\mathcal{Q}(X)$.
	
	Suppose that $\mathcal{Q}(X)$ is not a tree. Then it contains a cycle (possibly a loop, or two parallel edges), and a cycle traverses each of its edges exactly once. Write this cycle as $V_0,e_1,V_1,\dots,e_k,V_k=V_0$, where $e_i$ is the naked pair $\{x_i,y_i\}$ with $x_i\in V_{i-1}$ and $y_i\in V_i$, and the $e_i$ are distinct. Joining $y_i$ to $x_{i+1}$ (indices modulo $k$) by a path inside the component $V_i$ of $\mathcal{G}_B(X)$ produces a closed edge-path in $\mathcal{G}(X)$ which traverses $e_1$ exactly once, since the connecting paths use only edges of $\mathcal{G}_B(X)$. This contradicts Lemma~\ref{lem:naked}. Hence $\mathcal{Q}(X)$ is a tree, and its number of edges, $|E_X|$, is one less than its number of vertices, $\kappa(X)$.
\end{proof}

\begin{theorem}\label{thm:nonaked}
	Let $X$ be a minimal homotopically trivial finite $T_0$-space with
	$2\leq|X|\leq10$. Then $X$ has no naked pairs, and $\mathcal{G}_B(X)$ is
	connected.
\end{theorem}

\begin{proof}
	The space $X$ is connected and $H_1(\mathcal{K}(X))=0$, so by Theorem~\ref{thm:tree} it suffices to show that $\kappa(X)=1$. Suppose that $\kappa(X)\geq 2$.
	
	Let $V$ be a component of $\mathcal{G}_B(X)$. If $V$ contains no middle point, then $V$ has no edges, because every edge of $\mathcal{G}_B(X)$ has an endpoint in $B$; thus $V=\{x\}$ with $x\in C\cup A$. If $x=c\in C$, then $B_{>c}=\emptyset$, so every pair $(c,a)$ with $a\in A_{>c}$ is naked, and $|A_{>c}|\geq2$ by Lemma~\ref{lem:two-extremal}; hence $V$ has degree at least $2$ in $\mathcal{Q}(X)$. The case $x\in A$ is dual. If $V$ contains a middle point $b$, then $V\supseteq\{b\}\cup C_{<b}\cup A_{>b}$, so $|V|\geq5$ by Lemma~\ref{lem:two-extremal}.
	
	A tree with at least two vertices has at least two leaves. By the previous paragraph, two leaves $V$ and $V'$ of $\mathcal{Q}(X)$ contain middle points $b\in V$ and $b'\in V'$, and $10\geq|X|\geq|V|+|V'|\geq10$. Hence $X=V\sqcup V'$, the tree $\mathcal{Q}(X)$ has a single edge, so there is exactly one naked pair $(c_0,a_0)$, and $V=\{b\}\cup C_{<b}\cup A_{>b}$ and $V'=\{b'\}\cup C_{<b'}\cup A_{>b'}$ with $|C_{<b}|=|C_{<b'}|=2$. In particular $B=\{b,b'\}$ and $|C|=4$. Since $|C_{<b}|=2$, we may choose $c\in C_{<b}\setminus\{c_0\}$. The only middle point above $c$ is $b$, because $c<b'$ would give an edge of $\mathcal{G}_B(X)$ between $V$ and $V'$; and every $a\in A_{>c}$ satisfies $b<a$, because $(c,a)$ is transitive, as $c\neq c_0$. So $\hat F_c=\{b\}\cup A_{>c}$ has minimum $b$, and $c$ is an up-beat point, a contradiction.
\end{proof}

\begin{corollary}\label{cor:degree}
	Let $X$ be as in Theorem~\ref{thm:nonaked}. Then for all $c\in C$ and
	$a\in A$:
	\begin{enumerate}[label=\textup{(\roman*)}]
		\item $c<a$ if and only if $B_{>c}\cap B_{<a}\neq\emptyset$;
		\item the down-set $B_{<a}$ of $B$ is non-empty and has no maximum,
		i.e.\ it has at least two maximal elements;
		\item the up-set $B_{>c}$ of $B$ is non-empty and has no minimum,
		i.e.\ it has at least two minimal elements.
	\end{enumerate}
	In particular, if $B$ is an antichain, then $s(a)\geq2$ and $t(c)\geq2$ for
	all $a\in A$ and $c\in C$.
\end{corollary}

\begin{proof}
	Statement (i) is Theorem~\ref{thm:nonaked}. For (ii), $C_{<a}\neq\emptyset$ by Lemma~\ref{lem:two-extremal}, so $B_{<a}\neq\emptyset$ by (i). Suppose that $B_{<a}$ has a maximum $b$. Every point of $\hat U_a$ lies either in $B_{<a}$, and then it is at most $b$, or in $C_{<a}$, and then by (i) it lies below a point of $B_{<a}$, hence below $b$. Thus $b=\max\hat U_a$, and $a$ is a down-beat point, a contradiction. Statement (iii) is dual to (ii).
\end{proof}

\begin{remark}\label{rem:DU}
	For a finite poset $P$, let $\mathcal{D}(P)$ be the set of non-empty
	down-sets of $P$ with at least two maximal elements, and $\mathcal{U}(P)$
	the set of non-empty up-sets of $P$ with at least two minimal elements.
	For $X$ as in Theorem~\ref{thm:nonaked}, Corollary~\ref{cor:degree}
	says that $B_{<a}\in\mathcal{D}(B)$ for every $a\in A$ and
	$B_{>c}\in\mathcal{U}(B)$ for every $c\in C$. Since
	$\alpha(b)=|\{a\in A:b\in B_{<a}\}|$ and
	$\beta(b)=|\{c\in C:b\in B_{>c}\}|$, we obtain:
	\begin{enumerate}[label=\textup{(\alph*)}]
		\item every $b\in B$ belongs to some member of $\mathcal{D}(B)$ and to
		some member of $\mathcal{U}(B)$, by Lemma~\ref{lem:two-extremal};
		\item if $b\in B$ belongs to every member of $\mathcal{D}(B)$, then
		$A_{>b}=A$; if $b$ belongs to every member of $\mathcal{U}(B)$, then
		$C_{<b}=C$.
	\end{enumerate}
\end{remark}

\begin{remark}\label{rem:reconstruct}
	Let $X$ be a connected finite $T_0$-space with $|X|>1$ in which $B$ is an
	antichain and there are no naked pairs. Then $c<b$ if and only if
	$b\in B_{>c}$, $b<a$ if and only if $b\in B_{<a}$, and $c<a$ if and only if
	$B_{>c}\cap B_{<a}\neq\emptyset$. Hence $X$ is determined up to
	homeomorphism by the families $(B_{>c})_{c\in C}$ and $(B_{<a})_{a\in A}$
	of subsets of $B$, taken up to relabelling of $C$, of $A$ and of $B$.
\end{remark}

\begin{proposition}\label{prop:mincrit}
	Let $X$ be a connected finite $T_0$-space with $|X|>1$ in which $B$ is a
	non-empty antichain. Suppose that $X$ has no naked pairs, that
	$\beta(b)\geq2$ and $\alpha(b)\geq2$ for every $b\in B$, and that
	$t(c)\geq2$ and $s(a)\geq2$ for all $c\in C$ and $a\in A$. Then $X$ is
	minimal.
\end{proposition}

\begin{proof}
	For $b\in B$, the sets $\hat U_b=C_{<b}$ and $\hat F_b=A_{>b}$ are antichains with at least two elements, so $\hat U_b$ has no maximum and $\hat F_b$ has no minimum. For $c\in C$ we have $\hat U_c=\emptyset$ and $\hat F_c=B_{>c}\cup A_{>c}$. Every point of $A_{>c}$ lies above a point of $B_{>c}$, as there are no naked pairs, so the minimal elements of $\hat F_c$ are exactly the points of $B_{>c}$; there are at least two of them, so $\hat F_c$ has no minimum. The points of $A$ are treated dually.
\end{proof}

In the remaining sections, minimality and homotopical triviality are used only through the results of this section, apart from the equality $\chi(\mathcal{K}(X))=1$, which is substituted into Lemma~\ref{lem:euler3} in Propositions~\ref{prop:bx3-antichain} and~\ref{prop:bx4-antichain}, and the vanishing of $H_3(\mathcal{K}(X))$ in case~(h) of Proposition~\ref{prop:bx4-antichain}. The homotopical triviality of the spaces found is verified directly, by explicit collapses.

\section{The cases \texorpdfstring{$|B_X|\leq2$}{|B\_X|<=2}: no spaces exist}\label{sec:bx12}

\begin{theorem}\label{thm:bx1}
	No minimal homotopically trivial finite $T_0$-space with ten points has
	exactly one middle point.
\end{theorem}

\begin{proof}
	If $B=\{b\}$, then for every $a\in A$ the down-set $B_{<a}\subseteq\{b\}$ has at most one element, contradicting Corollary~\ref{cor:degree}(ii).
\end{proof}

\begin{theorem}\label{thm:bx2}
	No minimal homotopically trivial finite $T_0$-space with ten points has
	exactly two middle points.
\end{theorem}

\begin{proof}
	Let $B=\{b_1,b_2\}$. If $b_1$ and $b_2$ are comparable, every non-empty subset of $B$ has a maximum, contradicting Corollary~\ref{cor:degree}(ii). Otherwise $B$ is an antichain, so $\mathrm{h}(X)=2$, and Corollary~\ref{cor:degree} gives $s(a)=t(c)=2$ for all $a\in A$ and $c\in C$; that is, $A_{>b_1}=A_{>b_2}=A$ and $C_{<b_1}=C_{<b_2}=C$. Since $|A|\geq 2$ and $|C|\geq2$, this contradicts Lemma~\ref{lem:L33}.
\end{proof}

\section{The case \texorpdfstring{$|B_X|=3$}{|B\_X|=3}: exactly six spaces}\label{sec:bx3}

\begin{theorem}\label{thm:bx3}
	There are exactly six minimal homotopically trivial finite $T_0$-spaces
	$X$ with $|X|=10$ and $|B_X|=3$, up to homeomorphism: the spaces of
	Types~\textup{I}, \textup{II} and \textup{III} defined below, and their
	order-duals.
\end{theorem}

By Lemma~\ref{lem:size}, a space as in the theorem satisfies $\{|C|,|A|\}=\{3,4\}$. The proof of the theorem is given at the end of this section.

\subsection{The spaces of Types I, II and III}\label{subsec:types123}

In this subsection $C=\{c_1,c_2,c_3\}$, $B=\{b_1,b_2,b_3\}$ and $A=\{a_1,a_2,a_3,a_4\}$. Each type is the poset generated by the displayed relations between $C$ and $B$ and between $B$ and $A$; thus $c<a$ holds exactly when $c<b<a$ for some $b\in B$, and there are no naked pairs. Figures~\ref{I}, \ref{II} and~\ref{III} show the Hasse diagrams.

\paragraph*{Type~I.}
\begin{gather*}
	c_1<\{b_1,b_2\},\quad c_2<\{b_1,b_3\},\quad c_3<\{b_2,b_3\},\\
	b_1<\{a_1,a_2\},\quad b_2<\{a_1,a_3,a_4\},\quad b_3<A.
\end{gather*}
Here $B_{<a_1}=B$, $B_{<a_2}=\{b_1,b_3\}$ and $B_{<a_3}=B_{<a_4}=\{b_2,b_3\}$.

\begin{figure}[ht]
\centering
\begin{tikzpicture}[scale=0.63]\node (c1) at (0,0){$c_1$};\node (c2) at (2,0){$c_2$};\node (c3) at (4,0){$c_3$};\node (b1) at (0,2){$b_1$};\node (b2) at (2,2){$b_2$};\node (b3) at (4,2){$b_3$};\node (a1) at (0,4){$a_1$};\node (a2) at (2,4){$a_2$};\node (a3) at (4,4){$a_3$};\node (a4) at (6,4){$a_4$};\draw (c1)--(b1);\draw (c1)--(b2);\draw (c2)--(b1);\draw (c2)--(b3);\draw (c3)--(b2);\draw (c3)--(b3);\draw (b1)--(a1);\draw (b1)--(a2);\draw (b2)--(a1);\draw (b2)--(a3);\draw (b2)--(a4);\draw (b3)--(a1);\draw (b3)--(a2);\draw (b3)--(a3);\draw (b3)--(a4);\end{tikzpicture}
\caption{Type~I}
\label{I}
\end{figure}

\paragraph*{Type~II.}
\begin{gather*}
	c_1<\{b_1,b_2\},\quad c_2<\{b_1,b_3\},\quad c_3<\{b_2,b_3\},\\
	b_1<\{a_1,a_2,a_3\},\quad b_2<\{a_1,a_2,a_4\},\quad b_3<\{a_1,a_3,a_4\}.
\end{gather*}
Here $B_{<a_1}=B$, $B_{<a_2}=\{b_1,b_2\}$, $B_{<a_3}=\{b_1,b_3\}$ and $B_{<a_4}=\{b_2,b_3\}$.

\begin{figure}[ht]
\centering
\begin{tikzpicture}[scale=0.63]\node (c1) at (0,0){$c_1$};\node (c2) at (2,0){$c_2$};\node (c3) at (4,0){$c_3$};\node (b1) at (0,2){$b_1$};\node (b2) at (2,2){$b_2$};\node (b3) at (4,2){$b_3$};\node (a1) at (0,4){$a_1$};\node (a2) at (2,4){$a_2$};\node (a3) at (4,4){$a_3$};\node (a4) at (6,4){$a_4$};\draw (c1)--(b1);\draw (c1)--(b2);\draw (c2)--(b1);\draw (c2)--(b3);\draw (c3)--(b2);\draw (c3)--(b3);\draw (b1)--(a1);\draw (b1)--(a2);\draw (b1)--(a3);\draw (b2)--(a1);\draw (b2)--(a2);\draw (b2)--(a4);\draw (b3)--(a1);\draw (b3)--(a3);\draw (b3)--(a4);\end{tikzpicture}
\caption{Type~II}
\label{II}
\end{figure}

\paragraph*{Type~III.}
\begin{gather*}
	c_1<\{b_1,b_3\},\quad c_2<\{b_2,b_3\},\quad c_3<\{b_1,b_2,b_3\},\\
	b_1<\{a_1,a_2,a_3\},\quad b_2<\{a_1,a_2,a_4\},\quad b_3<\{a_3,a_4\}.
\end{gather*}
Here $B_{<a_1}=B_{<a_2}=\{b_1,b_2\}$, $B_{<a_3}=\{b_1,b_3\}$ and $B_{<a_4}=\{b_2,b_3\}$.

\begin{figure}[ht]
\centering
\begin{tikzpicture}[scale=0.63]\node (c1) at (0,0){$c_1$};\node (c2) at (2,0){$c_2$};\node (c3) at (4,0){$c_3$};\node (b1) at (0,2){$b_1$};\node (b2) at (2,2){$b_2$};\node (b3) at (4,2){$b_3$};\node (a1) at (0,4){$a_1$};\node (a2) at (2,4){$a_2$};\node (a3) at (4,4){$a_3$};\node (a4) at (6,4){$a_4$};\draw (c1)--(b1);\draw (c1)--(b3);\draw (c2)--(b2);\draw (c2)--(b3);\draw (c3)--(b1);\draw (c3)--(b2);\draw (c3)--(b3);\draw (b1)--(a1);\draw (b1)--(a2);\draw (b1)--(a3);\draw (b2)--(a1);\draw (b2)--(a2);\draw (b2)--(a4);\draw (b3)--(a3);\draw (b3)--(a4);\end{tikzpicture}
\caption{Type~III}
\label{III}
\end{figure}

In all three types $B$ is an antichain and every point of $C$ lies below every point of $A$. The degrees are as follows:
\[
\begin{array}{c|cccc}
	\text{Type} & (\beta_1,\beta_2,\beta_3) & (\alpha_1,\alpha_2,\alpha_3) & (t(c_1),t(c_2),t(c_3)) & (s(a_1),s(a_2),s(a_3),s(a_4))\\ \hline
	\text{I}   & (2,2,2) & (2,3,4) & (2,2,2) & (3,2,2,2)\\
	\text{II}  & (2,2,2) & (3,3,3) & (2,2,2) & (3,2,2,2)\\
	\text{III} & (2,2,3) & (3,3,2) & (2,2,3) & (2,2,2,2)
\end{array}
\]

\subsection{The middle level is an antichain}

\begin{proposition}\label{prop:bx3-antichain}
	Let $X$ be a minimal homotopically trivial finite $T_0$-space with
	$|X|=10$ and $|B_X|=3$. Then $B_X$ is an antichain, so $\mathrm{h}(X)=2$.
\end{proposition}

\begin{proof}
	Suppose that $B$ is not an antichain. By Proposition~\ref{prop:height}, $B$ contains no chain of three points. Hence two distinct relations between points of $B$ share either their smaller or their larger element (two relations among three points share a point, and if it were the larger element of one and the smaller element of the other, the two relations would form a chain of three points), and there are at most two relations, since three relations among three points would form a chain. After relabelling, the order on $B=\{b_1,b_2,b_3\}$ is therefore given by one of
	\[
	\text{(a) } b_1<b_2;\qquad
	\text{(b) } b_1<b_2 \text{ and } b_1<b_3;\qquad
	\text{(c) } b_1<b_3 \text{ and } b_2<b_3,
	\]
	all other pairs being incomparable. We use Remark~\ref{rem:DU}.
	
	\emph{Case (b).} An up-set of $B$ containing $b_1$ is $B$ itself, whose minimum is $b_1$. Hence $\mathcal{U}(B)=\{\{b_2,b_3\}\}$, and $b_1$ lies in no member of $\mathcal{U}(B)$, contradicting Remark~\ref{rem:DU}(a).
	
	\emph{Case (c).} This is the order-dual of case (b): $\mathcal{D}(B)=\{\{b_1,b_2\}\}$, and $b_3$ lies in no member of $\mathcal{D}(B)$.
	
	\emph{Case (a).} The non-empty down-sets of $B$ are $\{b_1\}$, $\{b_3\}$, $\{b_1,b_2\}$, $\{b_1,b_3\}$ and $B$, whose sets of maximal elements are $\{b_1\}$, $\{b_3\}$, $\{b_2\}$, $\{b_1,b_3\}$ and $\{b_2,b_3\}$. The non-empty up-sets are $\{b_2\}$, $\{b_3\}$, $\{b_1,b_2\}$, $\{b_2,b_3\}$ and $B$, whose sets of minimal elements are $\{b_2\}$, $\{b_3\}$, $\{b_1\}$, $\{b_2,b_3\}$ and $\{b_1,b_3\}$. Hence
	\[
	\mathcal{D}(B)=\bigl\{\{b_1,b_3\},B\bigr\},\qquad \mathcal{U}(B)=\bigl\{\{b_2,b_3\},B\bigr\}.
	\]
	By Remark~\ref{rem:DU}(b), $A_{>b_1}=A_{>b_3}=A$ and $C_{<b_2}=C_{<b_3}=C$; that is, $\alpha_1=\alpha_3=n$ and $\beta_2=\beta_3=m$. Moreover every $c\in C$ lies below every $a\in A$, through $b_3$, so $N_{CA}=mn$. Lemma~\ref{lem:euler3} with $R_B=\{(b_1,b_2)\}$ and $\chi(\mathcal{K}(X))=1$ gives
	\begin{multline*}
	1=m+n-mn+(\beta_1-1)(n-1)+(m-1)(\alpha_2-1)\\
	{}+(m-1)(n-1)-(\beta_1-1)(\alpha_2-1),
	\end{multline*}
	and since $m+n-mn+(m-1)(n-1)=1$, this simplifies to
	\[
	(\beta_1-1)(n-\alpha_2)+(m-1)(\alpha_2-1)=0.
	\]
	But $\beta_1\geq2$, $\alpha_2\leq n$, $\alpha_2\geq 2$ and $m\geq 3$, so the left-hand side is positive, a contradiction.
\end{proof}

\subsection{Classification in height two}

\begin{proposition}\label{prop:bx3-class}
	Let $X$ be a minimal homotopically trivial finite $T_0$-space with
	$|X|=10$, $|B_X|=3$, $|C|=3$ and $|A|=4$. Then $X$ is homeomorphic to the
	space of Type~\textup{I}, \textup{II} or~\textup{III}.
\end{proposition}

\begin{proof}
	By Proposition~\ref{prop:bx3-antichain}, $B=\{b_1,b_2,b_3\}$ is an antichain and $\mathrm{h}(X)=2$. By Corollary~\ref{cor:degree}, each of the sets $B_{>c}$ ($c\in C$) and $B_{<a}$ ($a\in A$) has at least two elements, so it is either $B$ or $B\setminus\{b_j\}$ for a unique $j$. For $j\in\{1,2,3\}$ put
	\[
	x_j=|\{c\in C:B_{>c}=B\setminus\{b_j\}\}|,\qquad y_j=|\{a\in A:B_{<a}=B\setminus\{b_j\}\}|,
	\]
	and let $x_0=3-(x_1+x_2+x_3)$ and $y_0=4-(y_1+y_2+y_3)$ be the numbers of $c\in C$ with $B_{>c}=B$ and of $a\in A$ with $B_{<a}=B$. Then $\beta_j=3-x_j$ and $\alpha_j=4-y_j$, so Lemma~\ref{lem:two-extremal} gives $x_j\leq1$ and $y_j\leq2$. By Theorem~\ref{thm:nonaked} and Remark~\ref{rem:reconstruct}, $X$ is determined up to homeomorphism by $(x_1,x_2,x_3;y_1,y_2,y_3)$, and relabelling $B$ permutes the indices $1,2,3$ simultaneously in the $x_j$ and in the $y_j$.
	
	Any two subsets of $B$ with at least two elements intersect, so $c<a$ for all $c\in C$ and $a\in A$ by Corollary~\ref{cor:degree}(i). Thus $N_{CA}=12$, and Corollary~\ref{cor:euler2} becomes
	\begin{equation}\label{eq:bx3-euler}
		\sum_{j=1}^{3}(2-x_j)(3-y_j)=6.
	\end{equation}
	Moreover, if $\{i,j,k\}=\{1,2,3\}$, then $C_{<b_i}\cap C_{<b_j}$ consists of the points $c$ with $B_{>c}\supseteq\{b_i,b_j\}$, so it has $x_0+x_k$ elements; similarly $|A_{>b_i}\cap A_{>b_j}|=y_0+y_k$. Lemma~\ref{lem:L33} therefore gives
	\begin{equation}\label{eq:bx3-L33}
		x_0+x_k\geq2\ \Longrightarrow\ y_0+y_k\leq1\qquad(k=1,2,3).
	\end{equation}
	We distinguish cases according to $x_1+x_2+x_3$.
	
	\emph{$x_1+x_2+x_3=0$.} Then~\eqref{eq:bx3-euler} reads $2\sum_j(3-y_j)=6$, so $y_1+y_2+y_3=6>4$, which is impossible.
	
	\emph{$x_1+x_2+x_3=1$.} After relabelling, $(x_1,x_2,x_3)=(1,0,0)$, and~\eqref{eq:bx3-euler} reads $y_1+2y_2+2y_3=9$. So $y_1$ is odd, hence $y_1=1$ and $y_2+y_3=4$, which gives $y_1+y_2+y_3=5>4$; impossible.
	
	\sloppy\emph{$x_1+x_2+x_3=2$.} After relabelling, $(x_1,x_2,x_3)=(1,1,0)$ and $x_0=1$, and~\eqref{eq:bx3-euler} reads $y_1+y_2+2y_3=6$. As $y_1+y_2+y_3\leq4$, this forces $y_3\geq2$, hence $y_3=2$, $y_1+y_2=2$ and $y_0=0$. Since $x_0+x_1=x_0+x_2=2$, \eqref{eq:bx3-L33} gives $y_1\leq1$ and $y_2\leq1$. Hence $(y_1,y_2,y_3)=(1,1,2)$.
	
	\emph{$x_1+x_2+x_3=3$.} Then $(x_1,x_2,x_3)=(1,1,1)$, $x_0=0$, and~\eqref{eq:bx3-euler} gives $y_1+y_2+y_3=3$. As each $y_j\leq2$, after relabelling (which does not change the $x_j$) we have $(y_1,y_2,y_3)=(2,1,0)$ or $(1,1,1)$.
	
	So $(x_1,x_2,x_3;y_1,y_2,y_3)$ is, up to relabelling of $B$, one of $(1,1,1;2,1,0)$, $(1,1,1;1,1,1)$ and $(1,1,0;1,1,2)$. Reading off the sets $B_{>c}$ and $B_{<a}$ from Section~\ref{subsec:types123}, these are the data of Types~I, II and III respectively.
\end{proof}

\subsection{Verification}

\begin{proposition}\label{prop:bx3-verify}\sloppy
	The spaces of Types~\textup{I}, \textup{II} and \textup{III} are minimal,
	have collapsible order complexes (so they are homotopically trivial), and
	are pairwise non-homeomorphic.
\end{proposition}

\begin{proof}
	\emph{Minimality.} In each type $B$ is a non-empty antichain, there are no naked pairs, and by the table in Section~\ref{subsec:types123} all the numbers $\beta_j$, $\alpha_j$, $t(c)$ and $s(a)$ are at least $2$. So each type is minimal by Proposition~\ref{prop:mincrit}.
	
	\sloppy\emph{Collapsibility.} The $2$-simplices of $\mathcal{K}(X)$ are the chains $c<b<a$ with $c\in C_{<b}$ and $a\in A_{>b}$. In each of the three types there are $10$ vertices, $27$ edges and $18$ such $2$-simplices; for Type~I they are:
	\[
	\begin{array}{@{}lllll@{}}
	\text{Through }b_1\colon & (c_1,b_1,a_1), & (c_1,b_1,a_2), & (c_2,b_1,a_1), & (c_2,b_1,a_2). \\
	\text{Through }b_2\colon & (c_1,b_2,a_1), & (c_1,b_2,a_3), & (c_1,b_2,a_4), & (c_3,b_2,a_1), \\
	 & (c_3,b_2,a_3), & (c_3,b_2,a_4). \\
	\text{Through }b_3\colon & (c_2,b_3,a_1), & (c_2,b_3,a_2), & (c_2,b_3,a_3), & (c_2,b_3,a_4), \\
	 & (c_3,b_3,a_1), & (c_3,b_3,a_2), & (c_3,b_3,a_3), & (c_3,b_3,a_4).
	\end{array}
	\]

	We write $(\sigma)\searrow\{\tau\}$ for the elementary collapse (Section~\ref{subsec:finite}) that removes the $2$-simplex $\sigma$, written as the chain of its vertices, together with its edge $\tau$, written as the set of its two vertices, where $\tau$ lies in no other remaining $2$-simplex. This shorthand, used only here and in Proposition~\ref{prop:bx4-verify}, records a single elementary collapse; performed in order, the collapses listed below give a collapse $\mathcal{K}(X)\searrow T$ in the sense of Section~\ref{subsec:finite}, where $T$ is the remaining spanning tree. That $\tau$ is free at each step is a finite check against the $2$-simplices remaining at that step. For example, at the first step for Type~I the only $2$-simplex containing both $c_1$ and $a_2$ is $(c_1,b_1,a_2)$, since $b_2\not<a_2$. The following sequences of eighteen elementary collapses remove all $2$-simplices.
	
	\textit{Type~I:}
	\[
	\begin{array}{@{}lll@{}}
	(c_1,b_1,a_2)\searrow\{c_1,a_2\}; & (c_1,b_1,a_1)\searrow\{c_1,b_1\}; & (c_2,b_1,a_2)\searrow\{b_1,a_2\}; \\
	(c_2,b_1,a_1)\searrow\{c_2,b_1\}; & (c_1,b_2,a_4)\searrow\{c_1,a_4\}; & (c_1,b_2,a_3)\searrow\{c_1,a_3\}; \\
	(c_1,b_2,a_1)\searrow\{c_1,b_2\}; & (c_3,b_2,a_4)\searrow\{b_2,a_4\}; & (c_3,b_2,a_3)\searrow\{b_2,a_3\}; \\
	(c_3,b_2,a_1)\searrow\{c_3,b_2\}; & (c_2,b_3,a_4)\searrow\{c_2,a_4\}; & (c_2,b_3,a_3)\searrow\{c_2,a_3\}; \\
	(c_2,b_3,a_2)\searrow\{c_2,a_2\}; & (c_2,b_3,a_1)\searrow\{c_2,b_3\}; & (c_3,b_3,a_4)\searrow\{c_3,a_4\}; \\
	(c_3,b_3,a_3)\searrow\{c_3,a_3\}; & (c_3,b_3,a_2)\searrow\{c_3,a_2\}; & (c_3,b_3,a_1)\searrow\{c_3,b_3\}.
	\end{array}
	\]
	Remaining spanning tree: $\{b_1,a_1\}$, $\{b_2,a_1\}$, $\{b_3,a_1\}$, $\{b_3,a_2\}$, $\{b_3,a_3\}$, $\{b_3,a_4\}$, $\{c_1,a_1\}$, $\{c_2,a_1\}$, $\{c_3,a_1\}$.
	
	\textit{Type~II:}
	\[
	\begin{array}{@{}lll@{}}
	(c_1,b_1,a_3)\searrow\{c_1,a_3\}; & (c_1,b_2,a_4)\searrow\{c_1,a_4\}; & (c_2,b_3,a_4)\searrow\{c_2,a_4\}; \\
	(c_3,b_3,a_3)\searrow\{c_3,a_3\}; & (c_3,b_2,a_2)\searrow\{c_3,a_2\}; & (c_2,b_1,a_2)\searrow\{c_2,a_2\}; \\
	(c_1,b_1,a_2)\searrow\{b_1,a_2\}; & (c_1,b_2,a_2)\searrow\{c_1,a_2\}; & (c_2,b_1,a_3)\searrow\{b_1,a_3\}; \\
	(c_2,b_3,a_3)\searrow\{c_2,a_3\}; & (c_3,b_3,a_4)\searrow\{b_3,a_4\}; & (c_3,b_2,a_4)\searrow\{b_2,a_4\}; \\
	(c_1,b_1,a_1)\searrow\{c_1,b_1\}; & (c_2,b_1,a_1)\searrow\{c_2,b_1\}; & (c_1,b_2,a_1)\searrow\{c_1,b_2\}; \\
	(c_2,b_3,a_1)\searrow\{c_2,b_3\}; & (c_3,b_2,a_1)\searrow\{c_3,b_2\}; & (c_3,b_3,a_1)\searrow\{c_3,b_3\}.
	\end{array}
	\]
	Remaining spanning tree: $\{b_1,a_1\}$, $\{b_2,a_1\}$, $\{b_2,a_2\}$, $\{b_3,a_1\}$, $\{b_3,a_3\}$, $\{c_1,a_1\}$, $\{c_2,a_1\}$, $\{c_3,a_1\}$, $\{c_3,a_4\}$.
	
	\textit{Type~III:}
	\[
	\begin{array}{@{}lll@{}}
	(c_1,b_1,a_2)\searrow\{c_1,a_2\}; & (c_3,b_1,a_2)\searrow\{b_1,a_2\}; & (c_3,b_2,a_2)\searrow\{c_3,a_2\}; \\
	(c_2,b_2,a_2)\searrow\{b_2,a_2\}; & (c_1,b_3,a_4)\searrow\{c_1,a_4\}; & (c_1,b_3,a_3)\searrow\{c_1,b_3\}; \\
	(c_1,b_1,a_1)\searrow\{c_1,a_1\}; & (c_3,b_1,a_1)\searrow\{b_1,a_1\}; & (c_3,b_2,a_1)\searrow\{c_3,a_1\}; \\
	(c_2,b_2,a_1)\searrow\{b_2,a_1\}; & (c_2,b_2,a_4)\searrow\{c_2,b_2\}; & (c_3,b_2,a_4)\searrow\{b_2,a_4\}; \\
	(c_3,b_3,a_4)\searrow\{c_3,a_4\}; & (c_2,b_3,a_4)\searrow\{b_3,a_4\}; & (c_1,b_1,a_3)\searrow\{c_1,b_1\}; \\
	(c_3,b_1,a_3)\searrow\{c_3,b_1\}; & (c_2,b_3,a_3)\searrow\{c_2,b_3\}; & (c_3,b_3,a_3)\searrow\{c_3,b_3\}.
	\end{array}
	\]
	Remaining spanning tree: $\{c_3,b_2\}$, $\{b_1,a_3\}$, $\{b_3,a_3\}$, $\{c_1,a_3\}$, $\{c_2,a_1\}$, $\{c_2,a_2\}$, $\{c_2,a_3\}$, $\{c_2,a_4\}$, $\{c_3,a_3\}$.
	
	In each case the remaining complex is a graph with $10$ vertices and $27-18=9$ edges, and it is connected, as the listed edges show; so it is a spanning tree. A tree collapses to a point by removing one leaf, together with its incident edge, at a time. Hence $\mathcal{K}(X)$ is collapsible, in particular contractible, and $X$ is homotopically trivial by McCord's theorem~\cite{McCord1966}.
	
	\emph{Non-homeomorphism.} A homeomorphism preserves $B$ and $A$, hence the multiset $\{\alpha_1,\alpha_2,\alpha_3\}$ of upper degrees of the middle points. By the table in Section~\ref{subsec:types123}, this multiset is $\{2,3,4\}$, $\{3,3,3\}$ and $\{2,3,3\}$ for Types~I, II and~III respectively.
\end{proof}

\begin{proof}[Proof of Theorem~\ref{thm:bx3}]
	Let $X$ be as in the theorem. If $|C|=3$ and $|A|=4$, then $X$ is homeomorphic to a space of Type~I, II or III by Proposition~\ref{prop:bx3-class}. Otherwise $|C|=4$ and $|A|=3$; then $X^{\mathrm{op}}$ is minimal and homotopically trivial with $|B_{X^{\mathrm{op}}}|=3$, three minimal points and four maximal points, so $X$ is homeomorphic to the order-dual of a space of Type~I, II or III.
	
	Conversely, these six spaces are minimal and homotopically trivial by Proposition~\ref{prop:bx3-verify}, since both properties are preserved by passing to the order-dual. Types~I, II and III are pairwise non-homeomorphic, hence so are their order-duals, because $Y\cong Z$ if and only if $Y^{\mathrm{op}}\cong Z^{\mathrm{op}}$. Finally, no type is homeomorphic to an order-dual type, since the former have three minimal points and the latter four. So there are exactly six spaces.
\end{proof}

\section{The case \texorpdfstring{$|B_X|=4$}{|B\_X|=4}: exactly four spaces}\label{sec:bx4}

\begin{theorem}\label{thm:bx4}
	There are exactly four minimal homotopically trivial finite $T_0$-spaces
	$X$ with $|X|=10$ and $|B_X|=4$, up to homeomorphism: the spaces of
	Types~\textup{IV}, \textup{V}, \textup{VI} and \textup{VII} defined below.
	Types~\textup{IV} and~\textup{VII} are self-dual, and the order-dual of
	Type~\textup{V} is homeomorphic to Type~\textup{VI}.
\end{theorem}

By Lemma~\ref{lem:size}, a space as in the theorem satisfies $|C|=|A|=3$. The proof of the theorem is given at the end of this section.

\subsection{The spaces of Types IV, V, VI and VII}\label{subsec:types4567}

In this subsection $C=\{c_1,c_2,c_3\}$, $B=\{b_1,b_2,b_3,b_4\}$ and $A=\{a_1,a_2,a_3\}$, and each type is the poset generated by the displayed relations, as in Section~\ref{subsec:types123}. Figures~\ref{IV}, \ref{V}, \ref{VI} and~\ref{VII} show the Hasse diagrams.

\paragraph*{Type~IV.}
\begin{gather*}
	c_1<\{b_1,b_3\},\quad c_2<\{b_2,b_4\},\quad c_3<\{b_1,b_2,b_3,b_4\},\\
	b_1,b_2<\{a_1,a_2\},\quad b_3,b_4<\{a_1,a_3\}.
\end{gather*}

\begin{figure}[ht]
\centering
\begin{tikzpicture}[scale=0.63]\node (c1) at (0,0){$c_1$};\node (c2) at (2,0){$c_2$};\node (c3) at (4,0){$c_3$};\node (b1) at (0,2){$b_1$};\node (b2) at (2,2){$b_2$};\node (b3) at (4,2){$b_3$};\node (b4) at (6,2){$b_4$};\node (a1) at (0,4){$a_1$};\node (a2) at (2,4){$a_2$};\node (a3) at (4,4){$a_3$};\draw (c1)--(b1);\draw (c1)--(b3);\draw (c2)--(b2);\draw (c2)--(b4);\draw (c3)--(b1);\draw (c3)--(b2);\draw (c3)--(b3);\draw (c3)--(b4);\draw (b1)--(a1);\draw (b1)--(a2);\draw (b2)--(a1);\draw (b2)--(a2);\draw (b3)--(a1);\draw (b3)--(a3);\draw (b4)--(a1);\draw (b4)--(a3);\end{tikzpicture}
\caption{Type~IV}
\label{IV}
\end{figure}

\paragraph*{Type~V.}
\begin{gather*}
	c_1<\{b_1,b_3\},\quad c_2<\{b_1,b_2,b_4\},\quad c_3<\{b_2,b_3,b_4\},\\
	b_1,b_2<\{a_1,a_2\},\quad b_3,b_4<\{a_1,a_3\}.
\end{gather*}

\begin{figure}[ht]
\centering
\begin{tikzpicture}[scale=0.63]\node (c1) at (0,0){$c_1$};\node (c2) at (3,0){$c_2$};\node (c3) at (6,0){$c_3$};\node (b1) at (1,2){$b_1$};\node (b2) at (3,2){$b_2$};\node (b3) at (5,2){$b_3$};\node (b4) at (7,2){$b_4$};\node (a1) at (1,4){$a_1$};\node (a2) at (4,4){$a_2$};\node (a3) at (7,4){$a_3$};\draw (c1)--(b1);\draw (c1)--(b3);\draw (c2)--(b1);\draw (c2)--(b2);\draw (c2)--(b4);\draw (c3)--(b2);\draw (c3)--(b3);\draw (c3)--(b4);\draw (b1)--(a1);\draw (b1)--(a2);\draw (b2)--(a1);\draw (b2)--(a2);\draw (b3)--(a1);\draw (b3)--(a3);\draw (b4)--(a1);\draw (b4)--(a3);\end{tikzpicture}
\caption{Type~V}
\label{V}
\end{figure}

\paragraph*{Type~VI.}
\begin{gather*}
	c_1<\{b_1,b_3\},\quad c_2<\{b_2,b_4\},\quad c_3<\{b_1,b_2,b_3,b_4\},\\
	b_1,b_2<\{a_1,a_2\},\quad b_3<\{a_1,a_3\},\quad b_4<\{a_2,a_3\}.
\end{gather*}

\begin{figure}[ht]
\centering
\begin{tikzpicture}[scale=0.63]\node (c1) at (0,0){$c_1$};\node (c2) at (2,0){$c_2$};\node (c3) at (4,0){$c_3$};\node (b1) at (0,2){$b_1$};\node (b2) at (2,2){$b_2$};\node (b3) at (4,2){$b_3$};\node (b4) at (6,2){$b_4$};\node (a1) at (0,4){$a_1$};\node (a2) at (2,4){$a_2$};\node (a3) at (4,4){$a_3$};\draw (c1)--(b1);\draw (c1)--(b3);\draw (c2)--(b2);\draw (c2)--(b4);\draw (c3)--(b1);\draw (c3)--(b2);\draw (c3)--(b3);\draw (c3)--(b4);\draw (b1)--(a1);\draw (b1)--(a2);\draw (b2)--(a1);\draw (b2)--(a2);\draw (b3)--(a1);\draw (b3)--(a3);\draw (b4)--(a2);\draw (b4)--(a3);\end{tikzpicture}
\caption{Type~VI}
\label{VI}
\end{figure}

\paragraph*{Type~VII.}
\begin{gather*}
	c_1<\{b_1,b_3\},\quad c_2<\{b_1,b_2,b_4\},\quad c_3<\{b_2,b_3,b_4\},\\
	b_1,b_2<\{a_1,a_2\},\quad b_3<\{a_1,a_3\},\quad b_4<\{a_2,a_3\}.
\end{gather*}

\begin{figure}[ht]
\centering
\begin{tikzpicture}[scale=0.63]\node (c1) at (0,0){$c_1$};\node (c2) at (3,0){$c_2$};\node (c3) at (6,0){$c_3$};\node (b1) at (1,2){$b_1$};\node (b2) at (3,2){$b_2$};\node (b3) at (5,2){$b_3$};\node (b4) at (7,2){$b_4$};\node (a1) at (1,4){$a_1$};\node (a2) at (4,4){$a_2$};\node (a3) at (7,4){$a_3$};\draw (c1)--(b1);\draw (c1)--(b3);\draw (c2)--(b1);\draw (c2)--(b2);\draw (c2)--(b4);\draw (c3)--(b2);\draw (c3)--(b3);\draw (c3)--(b4);\draw (b1)--(a1);\draw (b1)--(a2);\draw (b2)--(a1);\draw (b2)--(a2);\draw (b3)--(a1);\draw (b3)--(a3);\draw (b4)--(a2);\draw (b4)--(a3);\end{tikzpicture}
\caption{Type~VII}
\label{VII}
\end{figure}

In all four types $B$ is an antichain, $\beta_j=\alpha_j=2$ for every $j$, and every point of $C$ lies below every point of $A$. For a space with these properties and $b\in B$, let $\mu(b)$ denote the unique point of $C$ not below $b$, and $\nu(b)$ the unique point of $A$ not above $b$; then $t(c)=4-|\mu^{-1}(c)|$ and $s(a)=4-|\nu^{-1}(a)|$. Table~\ref{tab:munu} lists $\mu$ and $\nu$ for the four types, and the resulting multisets of degrees are:
\[
\begin{array}{c|cc}
	\text{Type} & \{t(c):c\in C\} & \{s(a):a\in A\}\\ \hline
	\text{IV}  & \{2,2,4\} & \{2,2,4\}\\
	\text{V}   & \{2,3,3\} & \{2,2,4\}\\
	\text{VI}  & \{2,2,4\} & \{2,3,3\}\\
	\text{VII} & \{2,3,3\} & \{2,3,3\}
\end{array}
\]

\begin{table}[ht]
	\centering
	\begin{tabular}{@{}ccccccccc@{}}
		\toprule
		Type & $\mu(b_1)$ & $\mu(b_2)$ & $\mu(b_3)$ & $\mu(b_4)$ & $\nu(b_1)$ & $\nu(b_2)$ & $\nu(b_3)$ & $\nu(b_4)$\\
		\midrule
		IV  & $c_2$ & $c_1$ & $c_2$ & $c_1$ & $a_3$ & $a_3$ & $a_2$ & $a_2$\\
		V   & $c_3$ & $c_1$ & $c_2$ & $c_1$ & $a_3$ & $a_3$ & $a_2$ & $a_2$\\
		VI  & $c_2$ & $c_1$ & $c_2$ & $c_1$ & $a_3$ & $a_3$ & $a_2$ & $a_1$\\
		VII & $c_3$ & $c_1$ & $c_2$ & $c_1$ & $a_3$ & $a_3$ & $a_2$ & $a_1$\\
		\bottomrule
	\end{tabular}
	\caption{For $b\in B$, $\mu(b)$ is the unique point of $C$ not below $b$, and $\nu(b)$ is the unique point of $A$ not above $b$.}
	\label{tab:munu}
\end{table}

\subsection{The middle level is an antichain}

\begin{table}[ht]
	\centering
	\begin{tabular}{@{}lll@{}}
		\toprule
		Case & Relations in $B$ & Excluded by\\
		\midrule
		(a) & $b_1<b_2$ & Lemma~\ref{lem:L31}, Euler count\\
		(b) & $b_1<b_3$, $b_2<b_4$ & Lemma~\ref{lem:L31}, Euler count\\
		(c) & $b_1<b_2$, $b_1<b_3$ & $\mathcal{D}(B)$, Lemma~\ref{lem:L31}, Euler count\\
		(d) & $b_2<b_1$, $b_3<b_1$ & order-dual of (c)\\
		(e) & $b_1<b_2$, $b_1<b_3$, $b_1<b_4$ & $\mathcal{U}(B)$\\
		(f) & $b_2<b_1$, $b_3<b_1$, $b_4<b_1$ & order-dual of (e)\\
		(g) & $b_1<b_3$, $b_1<b_4$, $b_2<b_4$ & $\mathcal{D}(B)$, $\mathcal{U}(B)$, Lemma~\ref{lem:L31}, Euler count\\
		(h) & $b_1<b_3$, $b_1<b_4$, $b_2<b_3$, $b_2<b_4$ & $\mathcal{D}(B)$, $\mathcal{U}(B)$, $H_3(\mathcal{K}(X))\neq0$\\
		\bottomrule
	\end{tabular}
	\caption{The eight partial orders on $B=\{b_1,b_2,b_3,b_4\}$, up to isomorphism, that are not antichains and contain no chain of three points, together with the arguments excluding them in Proposition~\ref{prop:bx4-antichain}. Pairs not listed are incomparable.}
	\label{tab:bx4posets}
\end{table}

\begin{proposition}\label{prop:bx4-antichain}
	Let $X$ be a minimal homotopically trivial finite $T_0$-space with
	$|X|=10$ and $|B_X|=4$. Then $B_X$ is an antichain, so $\mathrm{h}(X)=2$.
\end{proposition}

\begin{proof}
	By Lemma~\ref{lem:size}, $m=n=3$. Suppose that $B$ is not an antichain. By Proposition~\ref{prop:height}, $B$ contains no chain of three points. Hence the set $B^-$ of points of $B$ lying below some other point of $B$ and the set $B^+$ of points lying above some other point are disjoint and non-empty, every relation in $B$ goes from $B^-$ to $B^+$, and the relations form a bipartite graph on $B^-\sqcup B^+$ without isolated vertices; the remaining points of $B$ are incomparable with all others. If $|B^-|=1$ or $|B^+|=1$, this graph is a star, and $(|B^-|,|B^+|)=(1,1)$, $(1,2)$, $(2,1)$, $(1,3)$, $(3,1)$ give cases (a), (c), (d), (e), (f) of Table~\ref{tab:bx4posets}. If $|B^-|=|B^+|=2$, the graph consists of two disjoint edges, a path with three edges, or all four edges, giving cases (b), (g) and (h). Since $X^{\mathrm{op}}$ satisfies the same hypotheses as $X$ and has the reversed order on $B$, cases (d) and (f) follow from (c) and (e). We exclude the remaining six cases, using Lemma~\ref{lem:two-extremal} (so that $\beta_j,\alpha_j\in\{2,3\}$), Remark~\ref{rem:DU}, and Lemma~\ref{lem:euler3}. Since $\chi(\mathcal{K}(X))=1$ and $N_{CA}\leq mn=9$, the latter gives
	\begin{equation}\label{eq:bx4-euler}
		\sum_{j=1}^{4}(\beta_j-1)(\alpha_j-1)-\sum_{(b_i,b_k)\in R_B}(\beta_i-1)(\alpha_k-1)=N_{CA}-5\leq4.
	\end{equation}
	
	\emph{Case (a).} Here $\hat U_{b_1}=C_{<b_1}$ and $\hat U_{b_2}=C_{<b_2}\cup\{b_1\}$, so Lemma~\ref{lem:L31} gives $\beta_2+1\geq\beta_1+2\geq4$; hence $\beta_2=3$ and $\beta_1=2$. Dually, $\alpha_1=3$ and $\alpha_2=2$. The left-hand side of~\eqref{eq:bx4-euler} is then $2+2+(\beta_3-1)(\alpha_3-1)+(\beta_4-1)(\alpha_4-1)-1\geq5$, a contradiction.
	
	\emph{Case (b).} As in case (a), Lemma~\ref{lem:L31} applied to $b_1<b_3$ and to $b_2<b_4$ gives $\beta_1=\beta_2=2$, $\beta_3=\beta_4=3$, $\alpha_1=\alpha_2=3$ and $\alpha_3=\alpha_4=2$. The left-hand side of~\eqref{eq:bx4-euler} is $8-1-1=6$, a contradiction.
	
	\emph{Case (c).} A non-empty down-set of $B$ not containing $b_1$ contains neither $b_2$ nor $b_3$, so it is $\{b_4\}$, which has only one maximal element. Hence every member of $\mathcal{D}(B)$ contains $b_1$, and $\alpha_1=3$ by Remark~\ref{rem:DU}(b). Lemma~\ref{lem:L31} applied to $b_1<b_2$ and to $b_1<b_3$ gives $\beta_1=2$ and $\beta_2=\beta_3=3$, as in case (a). The left-hand side of~\eqref{eq:bx4-euler} is
	\begin{align*}
	&2+2(\alpha_2-1)+2(\alpha_3-1)+(\beta_4-1)(\alpha_4-1)\\
	&{}-(\alpha_2-1)-(\alpha_3-1)\\
	&=\alpha_2+\alpha_3+(\beta_4-1)(\alpha_4-1)\geq5,
	\end{align*}
	a contradiction.
	
	\emph{Case (e).} An up-set of $B$ containing $b_1$ is $B$ itself, whose minimum is $b_1$. Hence $b_1$ lies in no member of $\mathcal{U}(B)$, contradicting Remark~\ref{rem:DU}(a).
	
	\sloppy\emph{Case (g).} The non-empty down-sets of $B$ are $\{b_1\}$, $\{b_2\}$, $\{b_1,b_2\}$, $\{b_1,b_3\}$, $\{b_1,b_2,b_3\}$, $\{b_1,b_2,b_4\}$ and $B$, whose sets of maximal elements are $\{b_1\}$, $\{b_2\}$, $\{b_1,b_2\}$, $\{b_3\}$, $\{b_2,b_3\}$, $\{b_4\}$ and $\{b_3,b_4\}$. The non-empty up-sets are $\{b_3\}$, $\{b_4\}$, $\{b_3,b_4\}$, $\{b_2,b_4\}$, $\{b_2,b_3,b_4\}$, $\{b_1,b_3,b_4\}$ and $B$, whose sets of minimal elements are $\{b_3\}$, $\{b_4\}$, $\{b_3,b_4\}$, $\{b_2\}$, $\{b_2,b_3\}$, $\{b_1\}$ and $\{b_1,b_2\}$. Hence
	\[
	\mathcal{D}(B)=\bigl\{\{b_1,b_2\},\{b_1,b_2,b_3\},B\bigr\},\qquad
	\mathcal{U}(B)=\bigl\{\{b_3,b_4\},\{b_2,b_3,b_4\},B\bigr\},
	\]
	and Remark~\ref{rem:DU}(b) gives $\alpha_1=\alpha_2=3$ and $\beta_3=\beta_4=3$. Lemma~\ref{lem:L31} applied to $b_1<b_3$, where $\hat U_{b_3}=C\cup\{b_1\}$, gives $\beta_1=2$; applied to $b_2<b_4$, where $\hat F_{b_2}=A\cup\{b_4\}$, it gives $\alpha_4=2$. The left-hand side of~\eqref{eq:bx4-euler} is
	\[
	2+2(\beta_2-1)+2(\alpha_3-1)+2-(\alpha_3-1)-1-(\beta_2-1)=\beta_2+\alpha_3+1\geq5,
	\]
	a contradiction.
	
	\emph{Case (h).} A down-set of $B$ containing exactly one of $b_3,b_4$ has that point as its maximum, and a non-empty down-set containing neither is contained in $\{b_1,b_2\}$. So $\mathcal{D}(B)=\{\{b_1,b_2\},B\}$, and dually $\mathcal{U}(B)=\{\{b_3,b_4\},B\}$. Consequently $B_{>c}\in\{\{b_3,b_4\},B\}$ for every $c\in C$ and $B_{<a}\in\{\{b_1,b_2\},B\}$ for every $a\in A$, so that $C_{<b_1}=C_{<b_2}$ and $A_{>b_3}=A_{>b_4}$; by Lemma~\ref{lem:two-extremal} these sets have at least two elements. Choose distinct $c,c'\in C_{<b_1}$ and distinct $a,a'\in A_{>b_3}$, and put $\varepsilon(c)=\varepsilon(b_1)=\varepsilon(b_3)=\varepsilon(a)=1$ and $\varepsilon(c')=\varepsilon(b_2)=\varepsilon(b_4)=\varepsilon(a')=-1$. For $x\in\{c,c'\}$, $i\in\{1,2\}$, $k\in\{3,4\}$ and $y\in\{a,a'\}$, the chain $x<b_i<b_k<y$ is a $3$-simplex of $\mathcal{K}(X)$. Consider the $3$-chain
	\[
	z=\sum_{x,i,k,y}\varepsilon(x)\,\varepsilon(b_i)\,\varepsilon(b_k)\,\varepsilon(y)\,[x,b_i,b_k,y],
	\]
	which is non-zero because its sixteen simplices are distinct. The faces of these simplices obtained by deleting the vertex in a given position occur in pairs, differing only in the choice of the deleted vertex, with opposite coefficients; so $\partial z=0$. Since $\mathrm{h}(X)\leq3$, the complex $\mathcal{K}(X)$ has no $4$-simplices, so $z$ represents a non-zero class in $H_3(\mathcal{K}(X))$, a contradiction.
\end{proof}

\subsection{Classification in height two}

\begin{proposition}\label{prop:bx4-class}
	Let $X$ be a minimal homotopically trivial finite $T_0$-space with
	$|X|=10$ and $|B_X|=4$. Then $X$ is homeomorphic to the space of
	Type~\textup{IV}, \textup{V}, \textup{VI} or~\textup{VII}.
\end{proposition}

\begin{proof}
	By Lemma~\ref{lem:size} and Proposition~\ref{prop:bx4-antichain}, $m=n=3$ and $B$ is an antichain. Corollary~\ref{cor:euler2} gives $N_{CA}=5+\sum_{b\in B}(\beta(b)-1)(\alpha(b)-1)$. Each of the four summands is at least $1$ by Lemma~\ref{lem:two-extremal}, and $N_{CA}\leq 9$. Hence $N_{CA}=9$ and $\beta(b)=\alpha(b)=2$ for every $b\in B$. So $c<a$ for all $c\in C$ and $a\in A$, and the points $\mu(b)\in C$ and $\nu(b)\in A$ of Section~\ref{subsec:types4567} are defined for every $b\in B$:
	\[
	C_{<b}=C\setminus\{\mu(b)\},\qquad A_{>b}=A\setminus\{\nu(b)\}.
	\]
	We picture the pairs $(\mu(b),\nu(b))$ as marked cells in a $3\times3$ grid whose rows are indexed by $C$ and whose columns are indexed by $A$, and let $M$ be the set of marked cells. The conditions on $X$ translate as follows.
	\begin{enumerate}[label=\textup{(\arabic*)}]
		\item If $b\neq b'$ satisfy $\mu(b)=\mu(b')$ and $\nu(b)=\nu(b')$, then $|A_{>b}\cap A_{>b'}|=2$ and $|C_{<b}\cap C_{<b'}|=2$, contradicting Lemma~\ref{lem:L33}. So $b\mapsto(\mu(b),\nu(b))$ is injective, and $M$ has four cells.
		\item Since $t(c)=4-|\mu^{-1}(c)|$ and $s(a)=4-|\nu^{-1}(a)|$, Corollary~\ref{cor:degree} says that every row and every column contains at most two cells of $M$.
		\item For $c\in C$ and $a\in A$, $B_{>c}\cap B_{<a}=\{b\in B:\mu(b)\neq c,\ \nu(b)\neq a\}$. Since $c<a$, Corollary~\ref{cor:degree}(i) says that $M$ is not contained in the union of row $c$ and column $a$; we call such a union a \emph{cross}.
	\end{enumerate}
	Conversely, $X$ is determined up to homeomorphism by $M$ up to permutations of the rows and of the columns: the cells of $M$ correspond to the points of $B$ by~(1), each cell determines the sets $C_{<b}$ and $A_{>b}$, and $c<a$ for all $c$ and $a$.
	
	By (2), the row sums of $M$ are $2,2,0$ or $2,1,1$ in some order, and so are the column sums. We consider the four possibilities.
	
	\emph{Row sums $2,2,0$ and column sums $2,2,0$.} The two non-empty rows use the same two columns, so $M$ is a $2\times2$ block. A cross contains at most three cells of a $2\times2$ block, so (3) holds.
	
	\emph{Row sums $2,2,0$ and column sums $2,1,1$.} The two non-empty rows $c,c'$ use two-element sets of columns meeting in exactly one column, so
	\[
	M=\{(c,a),(c,a'),(c',a),(c',a'')\}\qquad\text{with } A=\{a,a',a''\}.
	\]
	A cross with row $r$ contains $M$ only if the cells of $M$ outside row $r$ lie in a single column; but the cells outside row $c$ lie in the columns $a,a''$, those outside row $c'$ lie in the columns $a,a'$, and those outside the third row lie in all three columns. So (3) holds, and $M$ is unique up to permutations.
	
	\emph{Row sums $2,1,1$ and column sums $2,2,0$.} This is the transpose of the previous case.
	
	\emph{Row sums $2,1,1$ and column sums $2,1,1$.} Let $c$ be the row with two cells, in the columns $a,a'$, and let the remaining two cells lie in the rows $c',c''$ and in the columns $x,y$. The column sums are those of the two cells in row $c$ increased by the contributions of $x$ and $y$; they have the form $2,1,1$ exactly when $\{x,y\}=\{a,a''\}$, $\{x,y\}=\{a',a''\}$ or $x=y=a''$, where $A=\{a,a',a''\}$. If $x=y=a''$, then $M$ lies in the cross of row $c$ and column $a''$, contradicting (3). Otherwise, after exchanging $a$ with $a'$ and $c'$ with $c''$ if necessary,
	\[
	M=\{(c,a),(c,a'),(c',a),(c'',a'')\}.
	\]
	A cross containing $(c,a)$ and $(c,a')$ must have row $c$, and then its column would have to contain both $(c',a)$ and $(c'',a'')$, which is impossible. So (3) holds, and $M$ is unique up to permutations.
	
	Thus $M$ is one of four configurations up to permutations of rows and columns, distinguished by their row and column sums. For each of Types~IV--VII, Table~\ref{tab:munu} gives the set of marked cells, and it satisfies (1)--(3): the four pairs $(\mu(b),\nu(b))$ are distinct, $t(c)\geq2$ and $s(a)\geq2$ for all $c$ and $a$, and there are no naked pairs. Since $t(c)=4-(\text{row sum of }c)$ and $s(a)=4-(\text{column sum of }a)$, comparing with the multisets of degrees in Section~\ref{subsec:types4567} shows that the four configurations above are those of Types~IV, VI, V and VII, in the order considered. As $X$ is determined by $M$, it is homeomorphic to one of these types.
\end{proof}

\subsection{Verification}

\begin{proposition}\label{prop:bx4-verify}
	The spaces of Types~\textup{IV}, \textup{V}, \textup{VI} and \textup{VII}
	are minimal, have collapsible order complexes (so they are homotopically
	trivial), and are pairwise non-homeomorphic.
\end{proposition}

\begin{proof}
	\emph{Minimality.} In each type $B$ is a non-empty antichain, there are no naked pairs, $\beta_j=\alpha_j=2$, and $t(c)\geq2$ and $s(a)\geq2$ by the table in Section~\ref{subsec:types4567}. So each type is minimal by Proposition~\ref{prop:mincrit}.
	
	\sloppy\emph{Collapsibility.} In each of the four types $\mathcal{K}(X)$ has $10$ vertices, $25$ edges and $16$ two-simplices $c<b<a$, read off from the displayed relations as in Proposition~\ref{prop:bx3-verify}, whose notation we keep. The following sequences of sixteen elementary collapses remove all $2$-simplices.
	
	\textit{Type~IV:}
	\[
	\begin{array}{@{}lll@{}}
	(c_1,b_1,a_2)\searrow\{c_1,a_2\}; & (c_1,b_3,a_3)\searrow\{c_1,a_3\}; & (c_2,b_2,a_2)\searrow\{c_2,a_2\}; \\
	(c_2,b_4,a_3)\searrow\{c_2,a_3\}; & (c_3,b_4,a_3)\searrow\{b_4,a_3\}; & (c_3,b_3,a_3)\searrow\{b_3,a_3\}; \\
	(c_3,b_1,a_2)\searrow\{b_1,a_2\}; & (c_3,b_2,a_2)\searrow\{b_2,a_2\}; & (c_1,b_1,a_1)\searrow\{c_1,b_1\}; \\
	(c_3,b_1,a_1)\searrow\{b_1,a_1\}; & (c_2,b_2,a_1)\searrow\{c_2,b_2\}; & (c_3,b_2,a_1)\searrow\{b_2,a_1\}; \\
	(c_1,b_3,a_1)\searrow\{c_1,b_3\}; & (c_3,b_3,a_1)\searrow\{b_3,a_1\}; & (c_2,b_4,a_1)\searrow\{c_2,b_4\}; \\
	(c_3,b_4,a_1)\searrow\{b_4,a_1\}.
	\end{array}
	\]
	Remaining spanning tree: $\{c_3,b_1\}$, $\{c_3,b_2\}$, $\{c_3,b_3\}$, $\{c_3,b_4\}$, $\{c_1,a_1\}$, $\{c_2,a_1\}$, $\{c_3,a_1\}$, $\{c_3,a_2\}$, $\{c_3,a_3\}$.
	
	\textit{Type~V:}
	\[
	\begin{array}{@{}lll@{}}
	(c_1,b_3,a_3)\searrow\{c_1,a_3\}; & (c_1,b_3,a_1)\searrow\{c_1,b_3\}; & (c_1,b_1,a_2)\searrow\{c_1,a_2\}; \\
	(c_1,b_1,a_1)\searrow\{c_1,b_1\}; & (c_2,b_1,a_1)\searrow\{b_1,a_1\}; & (c_2,b_1,a_2)\searrow\{b_1,a_2\}; \\
	(c_2,b_2,a_2)\searrow\{c_2,a_2\}; & (c_3,b_2,a_2)\searrow\{b_2,a_2\}; & (c_3,b_2,a_1)\searrow\{c_3,b_2\}; \\
	(c_2,b_2,a_1)\searrow\{b_2,a_1\}; & (c_3,b_3,a_3)\searrow\{b_3,a_3\}; & (c_3,b_4,a_3)\searrow\{c_3,a_3\}; \\
	(c_2,b_4,a_3)\searrow\{b_4,a_3\}; & (c_2,b_4,a_1)\searrow\{c_2,b_4\}; & (c_3,b_4,a_1)\searrow\{b_4,a_1\}; \\
	(c_3,b_3,a_1)\searrow\{c_3,a_1\}.
	\end{array}
	\]
	Remaining spanning tree: $\{c_1,a_1\}$, $\{c_2,b_1\}$, $\{c_2,a_1\}$, $\{c_2,b_2\}$, $\{c_3,a_2\}$, $\{b_3,a_1\}$, $\{c_3,b_3\}$, $\{c_2,a_3\}$, $\{c_3,b_4\}$.
	
	\textit{Type~VI:}
	\[
	\begin{array}{@{}lll@{}}
	(c_1,b_1,a_2)\searrow\{c_1,a_2\}; & (c_1,b_3,a_3)\searrow\{c_1,a_3\}; & (c_2,b_2,a_1)\searrow\{c_2,a_1\}; \\
	(c_2,b_4,a_3)\searrow\{c_2,a_3\}; & (c_3,b_1,a_2)\searrow\{b_1,a_2\}; & (c_3,b_3,a_3)\searrow\{b_3,a_3\}; \\
	(c_3,b_2,a_1)\searrow\{b_2,a_1\}; & (c_3,b_4,a_3)\searrow\{b_4,a_3\}; & (c_3,b_1,a_1)\searrow\{c_3,b_1\}; \\
	(c_3,b_3,a_1)\searrow\{c_3,b_3\}; & (c_3,b_2,a_2)\searrow\{c_3,b_2\}; & (c_3,b_4,a_2)\searrow\{c_3,b_4\}; \\
	(c_1,b_1,a_1)\searrow\{b_1,a_1\}; & (c_1,b_3,a_1)\searrow\{b_3,a_1\}; & (c_2,b_2,a_2)\searrow\{b_2,a_2\}; \\
	(c_2,b_4,a_2)\searrow\{b_4,a_2\}.
	\end{array}
	\]
	Remaining spanning tree: $\{c_1,b_1\}$, $\{c_1,b_3\}$, $\{c_2,b_2\}$, $\{c_2,b_4\}$, $\{c_1,a_1\}$, $\{c_2,a_2\}$, $\{c_3,a_1\}$, $\{c_3,a_2\}$, $\{c_3,a_3\}$.
	
	\textit{Type~VII:}
	\[
	\begin{array}{@{}lll@{}}
	(c_1,b_1,a_2)\searrow\{c_1,a_2\}; & (c_1,b_3,a_3)\searrow\{c_1,a_3\}; & (c_2,b_4,a_3)\searrow\{c_2,a_3\}; \\
	(c_2,b_1,a_2)\searrow\{b_1,a_2\}; & (c_3,b_3,a_3)\searrow\{b_3,a_3\}; & (c_3,b_4,a_3)\searrow\{b_4,a_3\}; \\
	(c_1,b_1,a_1)\searrow\{c_1,b_1\}; & (c_1,b_3,a_1)\searrow\{c_1,b_3\}; & (c_2,b_4,a_2)\searrow\{c_2,b_4\}; \\
	(c_2,b_2,a_2)\searrow\{c_2,a_2\}; & (c_3,b_3,a_1)\searrow\{c_3,b_3\}; & (c_3,b_4,a_2)\searrow\{c_3,b_4\}; \\
	(c_3,b_2,a_2)\searrow\{b_2,a_2\}; & (c_3,b_2,a_1)\searrow\{c_3,b_2\}; & (c_2,b_2,a_1)\searrow\{b_2,a_1\}; \\
	(c_2,b_1,a_1)\searrow\{b_1,a_1\}.
	\end{array}
	\]
	Remaining spanning tree: $\{c_2,b_1\}$, $\{c_2,b_2\}$, $\{b_3,a_1\}$, $\{b_4,a_2\}$, $\{c_1,a_1\}$, $\{c_2,a_1\}$, $\{c_3,a_1\}$, $\{c_3,a_2\}$, $\{c_3,a_3\}$.
	
	In each case the remaining graph has $10$ vertices and $25-16=9$ edges and is connected, so it is a spanning tree, which collapses to a point. Hence $\mathcal{K}(X)$ is collapsible, and $X$ is homotopically trivial by McCord's theorem~\cite{McCord1966}.
	
	\emph{Non-homeomorphism.} A homeomorphism preserves $C$ and $A$, hence the multisets $\{t(c):c\in C\}$ and $\{s(a):a\in A\}$. By the table in Section~\ref{subsec:types4567}, the four types have different pairs of multisets.
\end{proof}

\begin{proof}[Proof of Theorem~\ref{thm:bx4}]
	By Propositions~\ref{prop:bx4-class} and~\ref{prop:bx4-verify}, the spaces in question are exactly the four pairwise non-homeomorphic spaces of Types~IV--VII\@. The map $X\mapsto X^{\mathrm{op}}$ preserves minimality, homotopical triviality and $|B_X|$, so it permutes these four spaces; it exchanges the multisets $\{t(c)\}$ and $\{s(a)\}$. By the table in Section~\ref{subsec:types4567}, the order-dual of Type~IV has the invariants of Type~IV, the order-dual of Type~VII those of Type~VII, and the order-dual of Type~V those of Type~VI\@. Hence Types~IV and VII are self-dual and $\mathrm{V}^{\mathrm{op}}\cong\mathrm{VI}$.
\end{proof}

\begin{proof}[Proof of Theorem~\ref{thm:main}]
	Let $X$ be as in the theorem. By Lemma~\ref{lem:size}, $1\leq|B_X|\leq4$. Items (1) and (2) are Theorems~\ref{thm:bx1} and~\ref{thm:bx2}, and items (3) and (4) are Theorems~\ref{thm:bx3} and~\ref{thm:bx4}. In each of the ten spaces $B_X$ is an antichain and the height is two, by Propositions~\ref{prop:bx3-antichain} and~\ref{prop:bx4-antichain}, and the order complex is collapsible by Propositions~\ref{prop:bx3-verify} and~\ref{prop:bx4-verify}, since collapsibility is unaffected by passing to the order-dual. The last assertion is Theorem~\ref{thm:nonaked}.
\end{proof}

\section{Computational verification}\label{sec:computation}

\begin{remark}\label{rem:sagemath}
	The classification of Theorem~\ref{thm:main} admits an independent
	computational check, which is not used in any proof in this paper.
	Using SageMath~\cite{sagemath}, we enumerated all $2{,}567{,}284$ unlabelled
	posets on ten elements (OEIS~A000112), selected those without up-beat and
	down-beat points, and tested the resulting $7{,}929$ minimal posets (a count that includes disconnected posets, such as the discrete space with ten points) for
	acyclicity of the order complex $\mathcal{K}(X)$ by computing its
	integral homology groups. As a check on the pipeline, the same code
	reproduces the nine-point classification of Cianci and
	Ottina~\cite{CianciOttina2020}. The procedure is an exhaustive enumeration over a finite
	search space and involves no heuristics. The SageMath code is available from the authors upon request.
	
	Acyclicity is a complete test for homotopical triviality in this range. By a theorem of Matsushima and Tsukuda~\cite[Theorem~1.1]{MatsushimaTsukuda2024}, every connected finite $T_0$-space with at most twelve points is weakly homotopy equivalent to a wedge of spheres, and a wedge of spheres whose reduced integral homology vanishes is contractible. (Alternatively, by \cite[Theorem~5.7]{CianciOttina2018} the fundamental group of such a space is free, hence trivial when $H_1$ vanishes, and the Hurewicz and Whitehead theorems complete the argument.) So a ten-point space $X$ is homotopically trivial if and only if $\mathcal{K}(X)$ is connected and has vanishing reduced integral homology.
	
	The enumeration returns exactly ten minimal posets with acyclic order complex, distributed as follows:
	\[
	\begin{array}{c|c|c}
		|B_X| & \text{Theorem~\ref{thm:main}} & \text{SageMath} \\ \hline
		1 & 0 & 0 \\
		2 & 0 & 0 \\
		3 & 6 & 6 \\
		4 & 4 & 4 \\ \hline
		\text{Total} & 10 & 10
	\end{array}
	\]
	The computation corroborates the final count, not the individual steps of the proofs in Sections~\ref{sec:bx12}--\ref{sec:bx4}; the two are independent routes to the same result.
\end{remark}

\section{Conclusion}\label{sec:conclusion}

Combined with the nine-point classification of Cianci and Ottina~\cite[Theorem~3.7]{CianciOttina2020}, Theorem~\ref{thm:main} describes explicitly every homotopically trivial non-contractible minimal finite space with at most ten points, and the homotopical triviality of each of them is certified by an explicit collapse sequence. The count is confirmed independently by the enumeration of Section~\ref{sec:computation}.

The proofs rest on four ingredients: the height bound of Barmak and Minian and the size bounds of Cianci and Ottina (Proposition~\ref{prop:height} and Lemma~\ref{lem:size}); Theorem~\ref{thm:nonaked}, which shows that minimal homotopically trivial spaces with at most ten points have no naked pairs and thereby turns minimality into the degree conditions of Corollary~\ref{cor:degree}; Lemmas~\ref{lem:L31} and~\ref{lem:L33} of Cianci and Ottina; and the Euler characteristic count of Lemma~\ref{lem:euler3}.

\subsection{Naked pairs beyond ten points}\label{subsec:beyond}

In the proof of Theorem~\ref{thm:nonaked}, homotopical triviality enters only through the connectedness of $X$ and the vanishing of $H_1(\mathcal{K}(X))$, and the bound $|X|\leq10$ only through lower bounds for the sizes of the leaves $V$ of the tree $\mathcal{Q}(X)$. Using homotopical triviality in full, one obtains $|V|\geq9$ for every leaf. Consequently, the conclusion of Theorem~\ref{thm:nonaked} holds for all minimal homotopically trivial spaces with at most seventeen points, and Example~\ref{ex:eighteen} below shows that this is best possible.

\begin{proposition}\label{prop:eighteen}
	Let $X$ be a minimal homotopically trivial finite $T_0$-space that has a
	naked pair. Then every leaf of the tree $\mathcal{Q}(X)$ is a component of
	$\mathcal{G}_B(X)$ with at least nine points. Consequently $|X|\geq18$.
\end{proposition}

\begin{proof}
	By Theorem~\ref{thm:tree}, $\mathcal{Q}(X)$ is a tree; it has at least two vertices because $E_X\neq\emptyset$, and hence at least two leaves. Distinct leaves are disjoint subsets of $X$, so the second assertion follows from the first. Let $V$ be a leaf. Exactly one naked pair $(c_0,a_0)$ has an endpoint in $V$, and only one of its endpoints lies in $V$, since a tree has no loops. Passing to $X^{\mathrm{op}}$ if necessary, which leaves $\mathcal{G}_B(X)$ unchanged and reverses the naked pairs, we may assume that $c_0\in V$ and $a_0\notin V$. By the second paragraph of the proof of Theorem~\ref{thm:nonaked}, whose argument does not use the bound on $|X|$, $V$ contains a middle point. Put $C_V=C\cap V$, $B_V=B\cap V$ and $A_V=A\cap V$.
	
	\emph{Step~1.} If $x\in V$ and $y\in X$ are comparable, then $y\in V$ unless $\{x,y\}=\{c_0,a_0\}$. Indeed, if $x$ or $y$ lies in $B$, then $\{x,y\}$ is an edge of $\mathcal{G}_B(X)$. Otherwise one of them lies in $C$ and the other in $A$, and $\{x,y\}$ is either a transitive pair, whose two points lie in the same component of $\mathcal{G}_B(X)$ by the proof of Theorem~\ref{thm:tree}, or a naked pair with an endpoint in $V$, that is, $\{x,y\}=\{c_0,a_0\}$. Consequently, every point $x\in V\setminus\{c_0\}$ has the same sets $\hat U_x$ and $\hat F_x$ in the subspace $V$ as in $X$, and so is not a beat point of $V$. Moreover, for $b\in B_V$ the sets $C_{<b}$ and $A_{>b}$ are contained in $V$, and they have at least two points each by Lemma~\ref{lem:two-extremal}. As $V$ is a component of $\mathcal{G}_B(X)$ with more than one point, the subspace $V$ is connected, every point of $C_V$ lies below some point of $B_V$, and every point of $A_V$ lies above some point of $B_V$. In particular, $C_V$, $B_V$ and $A_V$ are the sets of minimal, middle and maximal points of $V$.
	
	\emph{Step~2: $V$ is homotopically trivial.} Let $r\colon X\to V\cup\{a_0\}$ be the identity on $V\cup\{a_0\}$ and send $X\setminus V$ to $a_0$. If $x<y$ in $X$, then by Step~1 either $x,y\in V$, or $x,y\notin V$, or $(x,y)=(c_0,a_0)$; in each case $r(x)\leq r(y)$. So $r$ is a retraction of $X$ onto $V\cup\{a_0\}$, and $V\cup\{a_0\}$ is homotopically trivial, since it is path-connected and its homotopy groups are retracts of those of $X$. By Step~1, $c_0$ is the only point of $V$ below $a_0$, so $a_0$ is a down-beat point of $V\cup\{a_0\}$. Hence $V$ is a strong deformation retract of $V\cup\{a_0\}$, and it is homotopically trivial.
	
	\emph{Step~3: if $B_{>c_0}$ has at least two minimal elements, then $|V|\geq9$.} By Step~1, the points of $V$ above $c_0$ are those of $B_{>c_0}\cup(A_{>c_0}\setminus\{a_0\})$, and $B_{>c_0}\neq\emptyset$. Every $a\in A_{>c_0}\setminus\{a_0\}$ lies above a point of $B_{>c_0}$, because $(c_0,a)$ is not the naked pair $(c_0,a_0)$ and is therefore transitive. So if $B_{>c_0}$ has at least two minimal elements, then $c_0$ is not a beat point of $V$, and by Step~1 the space $V$ is minimal. Being minimal with more than one point, $V$ is not contractible; since it is homotopically trivial by Step~2, Theorem~\ref{thm:CO} gives $|V|\geq9$.
	
	\emph{Step~4.} Let $a\in A_V$. Then $a\neq a_0$, so every pair $(c,a)$ with $c<a$ is transitive, and the argument in the proof of Corollary~\ref{cor:degree}(ii) shows that $B_{<a}$ is non-empty and has no maximum; by Step~1, $B_{<a}\subseteq B_V$. Dually, for $c\in C_V\setminus\{c_0\}$ the set $B_{>c}\subseteq B_V$ is non-empty and has no minimum. As $A_V\neq\emptyset$, it follows that $|B_V|\geq2$.
	
	We now show that $|V|\geq9$, distinguishing three cases.
	
	\emph{Case~1: $B_V$ is not an antichain.} Then there are $b,b'\in B_V$ such that $b'$ covers $b$, since every point lying between two points of $B_V$ belongs to $B_V$. By Step~1, the sets $\hat U_{b'}$ and $\hat F_b$ are contained in $V$, and they are disjoint because $b'$ covers $b$. By Lemmas~\ref{lem:L31} and~\ref{lem:two-extremal}, $|\hat U_{b'}|\geq|\hat U_b|+2\geq4$ and $|\hat F_b|\geq|\hat F_{b'}|+2\geq4$, so $|V|\geq8$. Suppose that $|V|=8$. Then $V=\hat U_{b'}\sqcup\hat F_b$, the set $\hat U_b=C_{<b}$ has exactly two points, and $\hat U_{b'}=\hat U_b\cup\{b,x\}$ for a single point $x\notin U_b$. Choose $c\in\hat U_b$ with $c\neq c_0$. By Step~4, $b$ is not the minimum of $B_{>c}$, so there is $z\in B_{>c}$ with $b\not\leq z$; then $z\notin\hat F_b$, so $z\in\hat U_{b'}$, and hence $z=x$. Thus $x\in B$ and $c<x$. The minimal points of $V$ are then the two points of $\hat U_b$, since $b$ and $x$ lie above $c$ and the points of $\hat F_b$ lie above $b$; so $C_{<x}=\hat U_b$ by Lemma~\ref{lem:two-extremal}, and $c_0$ lies below both $b$ and $x$. Now $x\not<b$, and $b\not<x$ because $\hat F_b\cap\hat U_{b'}=\emptyset$; moreover, every point below $b$ or below $x$ lies in $\hat U_b\subseteq C$, as $\hat U_x\subseteq\hat U_{b'}\setminus\{x\}=\hat U_b\cup\{b\}$ and $b\not<x$. Hence $b$ and $x$ are two minimal elements of $B_{>c_0}$, and Step~3 gives $|V|\geq9$, a contradiction. So $|V|\geq9$.
	
	\emph{Case~2: $B_V=\{b_1,b_2\}$ is an antichain.} Then $V$ has height two. By Step~4, every point of $A_V$ lies above both $b_1$ and $b_2$, and every point of $C_V\setminus\{c_0\}$ lies below both. Since $C_{<b_1}$ and $C_{<b_2}$ are subsets of $C_V$ with at least two points each, at least two points of $C_V$ lie below both $b_1$ and $b_2$: if $|C_V|\geq3$, any two points of $C_V\setminus\{c_0\}$ do, and if $|C_V|=2$, then $C_{<b_1}=C_{<b_2}=C_V$. As also $|A_V|\geq|A_{>b_1}|\geq2$, we obtain $|A_{>b_1}\cap A_{>b_2}|\geq2$ and $|C_{<b_1}\cap C_{<b_2}|\geq2$, which contradicts Lemma~\ref{lem:L33} applied to the homotopically trivial space $V$.
	
	\emph{Case~3: $B_V$ is an antichain with $k\geq3$ points.} Then $V$ has height two. If $|C_V|=2$, then every point of $B_V$ lies above both points of $C_V$, so $B_{>c_0}=B_V$ has $k\geq2$ minimal elements, and $|V|\geq9$ by Step~3. If $|A_V|=2$, then every point of $B_V$ lies below both points of $A_V$, hence so does every point of $C_V$, and the space $V$ has $N_{CA}=2|C_V|$. Corollary~\ref{cor:euler2}, applied to $V$, then gives
	\[
	2|C_V|=|C_V|+1+\sum_{b\in B_V}(\beta(b)-1),
	\]
	so $|C_V|\geq1+k\geq4$ and $|V|\geq4+3+2=9$. Otherwise $|C_V|\geq3$ and $|A_V|\geq3$, and $|V|\geq3+3+3=9$.
\end{proof}

\begin{example}\label{ex:eighteen}
	Let $Y$ and $Y'$ be disjoint connected minimal finite $T_0$-spaces with more than one point, choose $c\in\operatorname{mnl}(Y)$ and $a'\in\operatorname{mxl}(Y')$, and let $X$ be the union $Y\cup Y'$ with the orders of $Y$ and $Y'$ and the single additional relation $c<a'$. This is a partial order, since nothing lies below $c$ or above $a'$. No point lies strictly between $c$ and $a'$, so $(c,a')$ is a naked pair of $X$. The only sets $\hat U_x$ and $\hat F_x$ that change are $\hat F_c$, which gains the point $a'$, and $\hat U_{a'}$, which gains the point $c$. The set $\hat F_c\cap Y$ is non-empty, as $Y$ is connected and has more than one point, and $a'$ is incomparable with all its points; so $\hat F_c$ still has no minimum. Dually, $\hat U_{a'}$ has no maximum, and $X$ is minimal. Finally, the only chain of $X$ that meets both $Y$ and $Y'$ is $\{c,a'\}$, so $\mathcal{K}(X)$ is the union of $\mathcal{K}(Y)$, $\mathcal{K}(Y')$ and the edge $\{c,a'\}$. Hence $|\mathcal{K}(X)|$ is homotopy equivalent to the wedge $|\mathcal{K}(Y)|\vee|\mathcal{K}(Y')|$. In particular, if $Y$ and $Y'$ are homotopically trivial, then so is $X$. Taking for $Y$ and $Y'$ two copies of a nine-point space of Cianci and Ottina, we obtain a minimal homotopically trivial space with eighteen points and a naked pair.
\end{example}

Proposition~\ref{prop:eighteen} and Example~\ref{ex:eighteen} show that eighteen is the smallest number of points of a minimal homotopically trivial finite space with a naked pair. In particular, the conclusions of Theorem~\ref{thm:nonaked}, and those of Corollary~\ref{cor:degree}, whose proof uses only minimality and the absence of naked pairs, hold for every minimal homotopically trivial space with at most seventeen points. Homotopical triviality cannot be replaced here by connectedness together with the vanishing of $H_1(\mathcal{K}(X))$, which is all that the proof of Theorem~\ref{thm:nonaked} uses: the construction of Example~\ref{ex:eighteen}, applied to two copies of $\mathbb{S}^2S^0$, gives a minimal space $X$ with twelve points and a naked pair for which $|\mathcal{K}(X)|\simeq S^2\vee S^2$, so that $H_1(\mathcal{K}(X))=0$.

\subsection{Height three at eleven points}\label{subsec:eleven}

All nine- and ten-point examples have height two and an antichain middle level. This does not persist at eleven points. Let $Y$ be one of the two nine-point spaces of Cianci and Ottina, and consider its non-Hausdorff suspension $\mathbb{S}(Y)=Y\cup\{+,-\}$, which has eleven points. Its order complex is the simplicial suspension of the contractible complex $\mathcal{K}(Y)$, so $\mathbb{S}(Y)$ is homotopically trivial. It is also minimal. For $y\in Y$, the set $\hat U_y$ is unchanged, and $\hat F_y$ gains the two incomparable points $+$ and $-$, which lie above every point of $Y$; so $\hat F_y$ has a minimum in $\mathbb{S}(Y)$ only if it had one in $Y$. The points $\pm$ have $\hat F_{\pm}=\emptyset$ and $\hat U_{\pm}=Y$, which has no maximum because $Y$ is minimal with more than one point. The space $\mathbb{S}(Y)$ has height three, and its middle level $Y\setminus\operatorname{mnl}(Y)$ contains comparable points.

A classification of the minimal homotopically trivial spaces with eleven points therefore has to allow height three and comparable middle points, which Propositions~\ref{prop:bx3-antichain} and~\ref{prop:bx4-antichain} exclude at ten points, although by Proposition~\ref{prop:eighteen} such spaces still have no naked pairs. This classification is left open; it could be approached through an enumeration of the $46{,}749{,}427$ unlabelled posets on eleven elements (OEIS~A000112), along the lines of Section~\ref{sec:computation}.

\end{document}